\documentclass[11pt]{article}
\usepackage[utf8]{inputenc}
\usepackage{amsmath}
\usepackage{amsfonts}
\usepackage{amsthm}
\usepackage{amssymb}
\usepackage{mathtools}
\usepackage[T1]{fontenc}
\usepackage{pst-node}
\usepackage{tikz-cd}
\usepackage{caption}
\usepackage{bbm}
\usepackage{comment}

\usetikzlibrary{calc,backgrounds}

\usepackage{geometry}
\theoremstyle{plain}
\newtheorem{thm}{Theorem}[section]
\newtheorem{lem}[thm]{Lemma}
\newtheorem{prop}[thm]{Proposition}
\newtheorem{cor}[thm]{Corollary}

\newtheorem*{claim}{Claim}

\theoremstyle{definition}
\newtheorem{defn}[thm]{Definition}
\newtheorem{exmp}[thm]{Example}

\theoremstyle{remark}
\newtheorem*{rem}{Remark}

\newcommand{\re}{\mathbb{R}}

\newcommand{\F}{\mathbb{F}}
\newcommand{\ze}{\mathbb{Z}}

\newcommand{\na}{\mathbb{N}}

\renewcommand{\Im}{\mathrm{Im} \, }

\newcommand{\cd}{\mathrm{cd}}
\newcommand{\vcd}{\mathrm{vcd}}

\newcommand{\Lk}{\mathrm{Lk}}

\newcommand{\Id}{\mathrm{Id}}

\newcommand{\Cdim}{\mathrm{Confdim}}
\newcommand{\Hdim}{\mathrm{Hausdim}}

\renewcommand{\a}{\alpha}

\newcommand{\C}{\mathcal{C}}
\newcommand{\g}{\gamma}
\newcommand{\G}{\Gamma}

\newcommand{\D}{\Delta}

\renewcommand{\l}{\lambda}

\newcommand{\s}{\sigma}
\renewcommand{\S}{\Sigma}
\renewcommand{\t}{\tau}

\title{Top degree $\ell^p$-homology and conformal dimension of buildings}

\author{Antonio López Neumann}

\date{}

\begin{document}

\maketitle

\begin{abstract}
    For a non-compact finite thickness building whose Davis apartment is an orientable pseudomanifold, we compute the supremum of the set of $p>1$ such that its top dimensional reduced $\ell^p$-cohomology is nonzero. We adapt the non-vanishing assertion of this result to any finite thickness building using the Bestvina realization. Using similar techniques, we generalize bounds obtained by Clais on the conformal dimension of some Gromov-hyperbolic buildings to any such building.
    
    \vspace*{2mm} \noindent{2010 Mathematics Subject Classification: 20J06, 20F67, 22E41, 30L10, 43A15, 51E24.}

%









\vspace*{2mm} \noindent{Keywords and Phrases: $\ell^p$-cohomology, buildings, cohomological dimension, conformal dimension.}

\end{abstract}

\tableofcontents

\section{Introduction}

Buildings were introduced by Tits first as incidence geometries allowing to show simplicity of some families of groups. Later generalizations turned them into non-Archimedean analogues of symmetric spaces. A building is obtained by patching together under some incidence conditions, many copies of a same simplicial complex, which is constructed from a Coxeter system and that we call its abstract Coxeter complex. 

Buildings may be viewed as metric spaces, and as such we may study some of their quasi-isometry invariants. Here we are interested in their $\ell^p$-cohomology. Many variants of $\ell^p$-cohomology exist today: de Rham $L^p$-cohomology, simplicial $\ell^p$-cohomology or (continuous) group $\ell^p$-cohomology are some of them. These are quasi-isometry invariants popularized by Gromov in \cite{gromov} and intensively studied in hyperbolic settings.

A first intuition on $\ell^p$-cohomology is that its dependence on $p$ should behave like the dependence on $\a$ of $\a$-Hausdorff measures for a given metric space. Indeed, the first $\ell^p$-cohomology group increases as $p$ grows to infinity, thus there exists a critical $p$ for which it starts to be nonzero. This critical exponent is a numerical quasi-isometry invariant first introduced by Pansu \cite{pansu89} which can be thought of as some sort of Hausdorff dimension. In higher degrees the situation may be more subtle, as suggested by the description of the vanishings of $L^p$-cohomology of the real hyperbolic space $\mathbb{H}^n$ for $n \geq 3$ \cite{pansu08}. 


We are interested in top dimensional simplicial $\ell^p$-cohomology of buildings. In maximal degree, it is easier to study simplicial $\ell^p$-homology than $\ell^p$-cohomology, because the former is not a quotient space. Then duality allows us to recover reduced $\ell^p$-cohomology. Top degree $\ell^p$-homology has the same behaviour as the first $\ell^p$-cohomology group: it increases with $p$ so there is again a critical exponent for which this space starts to be nonzero. In the case where the Davis realization $\S_D$ of the Coxeter complex of the building is an orientable pseudomanifold, we compute this critical exponent in terms of the combinatorial data describing the building: its Weyl group $(W,S)$ and its thickness vector $\mathbf{q} + 1$. For $w \in W$, let $q_w$ be the number of galleries of type $w$ starting from a given chamber. We call 
\begin{equation*}
    e_\mathbf{q}(W) = \limsup_n \frac{1}{n} \log | \{ w \in W, q_w \leq e^n \} |
\end{equation*}
the weighted exponential growth rate of combinatorial balls in the building $X$ (see Section 2.3).

\begin{thm}(See Theorem \ref{criticalPM} in the text)
Let $(W,S)$ be a Coxeter system such that $\S_D(W,S)$ is an orientable pseudomanifold. Let $X_D$ be the Davis realization of a regular building of type $(W,S)$ and thickness vector $\mathbf{q} + 1$, with $\mathbf{q} \geq 2$. Let $n = \dim X_D$. Then we have:
\begin{align*}
    1 + e_\mathbf{q}(W) = \inf \{ p > 1 \, | \, \ell^p H_n(X_D) \neq \{ 0 \} \}, \\
    1 + e_\mathbf{q}(W)^{-1} = \sup \{ p > 1 \, | \, \ell^p \overline{H^n}(X_D) \neq  \{ 0 \}  \}.
\end{align*}
\end{thm}

This result is an extension to $\ell^p$-cohomology of \cite{dymarafundclass}, with essentially the same proof (the main idea can be traced back at least to \cite[p. 221]{gromov}). We obtain two corollaries from it. The first concerns $\ell^p$-cohomology of affine buildings and semisimple groups over non-Archimedean local fields.
 
\begin{cor}(See Corollary \ref{affineVSnon-affine}.1 in the text) Let $X$ be an affine building of dimension $n$ and finite thickness. We have $\ell^p \overline{H^n}(X) \neq \{0\}$ for all $p > 1$. In particular, any semisimple group $G$ of rank $n$ over a non-Archimedean local field satisfies $\overline{H_\mathrm{ct}^n}(G, L^p(G)) \neq \{0\}$ for all $p > 1$. 
\end{cor} 
 
In \cite[253]{gromov}, Gromov asked: for a semisimple group $G$ of rank $n$ over a local field, do we have $H_\mathrm{ct}^n(G, L^p(G))  \neq \{0\}$ at least for some $1 < p< \infty$ (where $L^p(G)$ denotes the right regular representation)? In the real case, this non-vanishing is shown for large $p>1$ in \cite{bourdon-remy-non-vanishing}, the behaviour of $H_\mathrm{ct}^n(G, L^p(G)) $ when $p>1$ is close to 1 is still unknown. In the non-Archimedean case, Gromov showed that $\overline{H_\mathrm{ct}^n}(G, L^p(G)) \neq \{0\}$ for $1<p \leq 2$ \cite[p. 255]{gromov} and expected this non-vanishing to hold for all $1<p< \infty$, which is the content of the previous corollary.

Our second corollary is the description of the set of $p$'s for which the second $\ell^p$-cohomology group of a cocompact Fuchsian building $X$ vanishes. The same question for the first $\ell^p$-cohomology group was already solved in \cite{bourdon-pajot}: the infimal $p$ for which this space starts to be nonzero is the conformal dimension $\Cdim(\partial X)$ of the boundary $\partial X$ (for the definition of $\Cdim(\partial X)$ see \ref{def-confdim}). Our theorem implies that conformal dimension is also a critical exponent for the second $\ell^p$-cohomology groups, but with a somewhat opposite behaviour. Indeed, \cite{bourdon-fuchsien}, Bourdon shows that for these buildings one has $\Cdim(\partial X) =  1 + e_\mathbf{q}(W)^{-1}$, so our result becomes $\ell^p H^2(X) \neq 0$ for $1<p< \Cdim(\partial X)$. Moreover \cite{bourdon-lp-degre-superieur}, for $p > \Cdim(\partial X)$, we have $\ell^p H^2(X) = 0$. By putting all of these results together, we obtain the following description of vanishings of $\ell^p$-cohomology of cocompact Fuchsian buildings in terms of $p$ in degrees 1 and 2.

\begin{cor}(See Corollary \ref{applicationFuchsian} in the text)
Let $X$ be the Davis realization of a cocompact Fuchsian building of type $(W,S)$ and thickness vector $\mathbf{q} + 1$, with $\mathbf{q} \geq 2$. \\
$\bullet$ For $p < \Cdim(\partial X) =  1 + e_\mathbf{q}(W)^{-1}$, we have 
\begin{center}
    $\ell^p H^1(X) = \{0\}$ and $\ell^p \overline{H^2}(X) \neq \{0\}$. 
\end{center}
$\bullet$ For $p > \Cdim(\partial X) =  1 + e_\mathbf{q}(W)^{-1}$, we have 
\begin{center}
    $\ell^p H^1(X) \neq \{0\}$ and $\ell^p H^2(X) = \{0\}$.
\end{center}
\end{cor}

The critical exponent for $\ell^p$-homology $1 + e_\mathbf{q}(W)$ is finite when $\mathbf{q} \geq 2$, so in top dimension we have non-vanishing of $\ell^p$-homology for large $p$. We may ask if this holds without the pseudomanifold assumption on $\S_D$. The problem is that the dimension of the Davis realization has no direct relation with cohomology and there are examples where we have uniform vanishing for all $p>1$ in degree $\dim \S_D$. This happens because of the local topology of the Davis chamber. This local problem is solved by considering a different realization $\S_B$, called the Bestvina realization, whose dimension is always equal to the virtual cohomological dimension $\vcd_\re(W)$ of $W$ over $\re$ (which may be strictly smaller than the dimension of $\S_D$). The Bestvina realization has poor metric properties but it is well suited for cohomological computations. We show non-vanishing of $\ell^p$-homology in degree $ \dim X_B = \vcd_\re(W)$ for large $p$.

\begin{thm}(See Theorem \ref{BestvinaNonVanishing} in the text)
Let $X_B$ be the Bestvina realization of a regular building of type $(W,S)$ and thickness vector $\mathbf{q} +1$, with $\mathbf{q} \geq 2$. Let $d =  \dim {X_B}= \vcd_\re(W)$. For all $p > 1 + e_\mathbf{q}(W)$, we have $\ell^p H_d(X_B) \neq  \{ 0 \} $.
\end{thm}
The Davis and Bestvina realizations are quasi-isometric, so this non-vanishing result also holds for the Davis realization. We do not know how to compute the infimal $p$ for which top degree $\ell^p$-homology starts to be nonzero in this more general setting.

These ideas and particularly the formula of the critical exponent in top dimension for $\ell^p$-cohomology led us to a further study of conformal dimension of Gromov-hyperbolic buildings. 
Few is known about these conformal dimensions. The only known exact computation is $\mathrm{Confdim}(\partial X) = 1 +  e_\mathbf{q}(W)^{-1}$ \cite{bourdon-fuchsien} for a Fuchsian building $X$ of type $(W,S)$. Our previous non-vanishing result can be compared with \cite[Théorème A]{bourdon-lp-degre-superieur}, to obtain the inequality:
\begin{equation*}
    \frac{\Cdim (\partial X)}{\vcd_\re(W) -1} \geq  1 + e_\mathbf{q}(W)^{-1}.
\end{equation*}
We may compare this to some bounds obtained by Clais for the conformal dimension of hyperbolic buildings arising from right-angled Coxeter groups \cite{clais-conformal}. For these buildings, Clais' lower bound of $\mathrm{Confdim}(\partial X)$ is sharper than our inequality obtained using top-dimensional $\ell^p$-cohomology. We adapt some of our techniques to the first $\ell^p$-cohomology groups and, combined with a result from \cite{bourdon-kleiner} relating conformal dimension and the first $\ell^p$-cohomology groups, we generalize Clais' bounds to arbitrary Gromov-hyperbolic buildings.
 
\begin{thm}\label{ConfdimBoundsIntro} (See Theorem \ref{ConfdimBounds} in the text) Let $(W,S)$ be a Gromov-hyperbolic Coxeter system, $\S$ the Davis complex of $(W,S)$ and $X$ the Davis realization of a building with Weyl group $(W,S)$ and thickness vector $\mathbf{q} +1$. Let $d_\mathbf{q}$ be a visual metric on $\partial \S$ induced by a combinatorial distance $| \cdot - \cdot|_\mathbf{q}$ on the dual graph of $X$.
We have:
\begin{equation*}
     \mathrm{Confdim}(\partial \S) (1 + e_\mathbf{q}(W)^{-1}) \leq \mathrm{Confdim}(\partial X) \leq \Hdim(\partial \S, d_\mathbf{q}) (1 + e_\mathbf{q}(W)^{-1}).
\end{equation*}
\end{thm}

The main idea that holds the different sections together is the following: there are natural maps $\rho$ retracting a building $X$ onto its corresponding Coxeter complex $\S$. Naturally, there is a pushforward $\rho_* : C_k^{\mathrm{lf}}(X; \re) \to C_k^{\mathrm{lf}}(\S; \re)$ sending formal chains from the building to the Coxeter complex. The point is that we can also define a pullback $\rho^* : C_k^{\mathrm{lf}}(\S) \to C_k^{\mathrm{lf}}(X)$ on formal chains that commutes with the boundary operator. This is done using a harmonicity formula, which is basically a dual version of the Steinberg representation. Then we use convexity estimates (in particular Jensen's inequality) and growth rates to decide on finiteness of $\ell^p$-norms.

Similar ideas can be applied to the first $\ell^p$-cohomology group. Indeed, cochains can be seen as functions $f : X^{(0)} \to \re$ on vertices of $X$ whose simplicial differential is $\ell^p$. Here the pullback $\rho^*$ is naturally defined as precomposition by $\rho$ and pushforward $\rho_*$ of functions is defined by taking averages on preimages of $\rho$. The main question here is: given $f : \S \to \re$ with a control on the infimal $p>1$ such that $||df||_p < \infty$, what can be said about an $r > 1$ such that
\begin{equation*}
     ||d(\rho^*f)||_r^r = \sum_{\s \in \S^{(1)}} |\rho^{-1}(\s)| |df(\s)|^r < \infty?
\end{equation*}
Again, convexity gives a partial answer to this question. More precisely, we use a pressure like function inspired on thermodynamical formalism, that yields the lower bound of Theorem \ref{ConfdimBoundsIntro}.

The article is organized as follows. Section 2 contains most of the preliminaries. We first introduce simplicial $\ell^p$-homology and $\ell^p$-cohomology. Then we recall some elementary building theory: combinatorial buildings, their geometric realizations and growth rates. We introduce the maps induced on formal chains by a retraction. Finally, we give a brief discussion on the Davis realization of a building. All other three sections are independent from each other. Section 3 introduces the family of buildings whose apartments are orientable pseudomanifolds and contains the computation of the critical exponent on top dimensional $\ell^p$-homology. Section 4 reviews \cite{bestvina-cohom-dim} in order to give an $\re$-acyclic version of the Bestvina chamber. Then we show non-vanishing of $\ell^p$-homology in top degree for the Bestvina realization. Section 5 first introduces the notion of conformal dimension of Gromov-hyperbolic spaces and discusses its connection with $\ell^p$-cohomology. Subsequently we prove our generalization of Clais' bounds.

\textbf{Conventions.} In what follows $p$ is a real number $>1$ and $(W,S)$ denotes a Coxeter system where $W$ is infinite and $S$ is finite. An affine Coxeter system is the direct product of finite Coxeter systems and at least one irreducible affine Coxeter system. Any other infinite Coxeter system is said to be non-affine. By the strong Tits alternative for Coxeter groups \cite[17.2.1]{davisbook}, this is the same as saying that affine Coxeter systems are infinite Coxeter systems of polynomial growth and non-affine Coxeter systems are those of exponential growth.

\textbf{Acknowledgements} 
This research is part of the author's PhD thesis, which was funded by the École Doctorale de Mathématiques Hadamard, at École polytechnique. 
I want to thank my PhD advisors, Marc Bourdon and Bertrand Rémy, for their constant support, insight, as well as for many of their comments that helped me write this article. I also want to thank the reviewer for multiple comments that improved the exposition of the paper.

\section{$\ell^p$-cohomology and buildings}

We first introduce $\ell^p$-homology and $\ell^p$-cohomology for complexes of bounded geometry and review some fundamental properties. We then review some combinatorial theory of buildings and adopt a uniform point of view on their geometric realizations (as presented in \cite[Chapter 18]{davisbook}).

\subsection{$\ell^p$-homology and $\ell^p$-cohomology}

We define $\ell^p$-homology and $\ell^p$-cohomology as presented in \cite{bourdon-survol}.

First, let $X$ be a simplicial complex equipped with a metric so that it becomes a length space. We say that such a complex $X$ has \textit{bounded geometry} if: \\
(i) there exists a constant $C>0$ such that every simplex of $X$ has diameter $\leq C$, \\
(ii) there is a function $N : [0, \infty [ \to \na$ such that for all $r > 0$, every ball of radius $r$ contains at most $N(r)$ simplices of $X$.

Now we define $\ell^p$-homology and $\ell^p$-cohomology for complexes of bounded geometry. For $k \in \na$, we denote by $X^{(k)}$ the set of $k$-simplices of $X$. For $1 < p < \infty$, we define:
\begin{equation*}
    \ell^p C_k (X) = \{ \sum_{\s \in X^{(k)}} a_\s \s, \, (a_\s)_{\s \in X^{(k)}} \in \ell^p(X^{(k)})  \},
\end{equation*} 
\begin{equation*}
    \ell^p C^k(X) = \{ \omega : X^{(k)}\to \re,  \, \omega \in \ell^p(X^{(k)})  \}.
\end{equation*}

The standard boundary operators $\partial_k$ and $d_k$ are defined as usual (after choosing an ordering of the vertices of $X$ or equivalently an orientation on simplices of $X$). Because the complex $X$ has bounded geometry, they define bounded operators:
\begin{equation*}
    \partial_k : \ell^p C_k (X) \to \ell^p C_{k-1} (X), \qquad d_k : \ell^p C^k(X) \to \ell^p C^{k+1}(X).
\end{equation*}
These two operators are related by the simplicial version of Stokes' theorem: for $c \in \ell^p C^k(X)$ and $\s \in \ell^p C_{k+1} (X)$ we have $d_k c (\s) = c(\partial_{k+1} \s)$.

\begin{defn} Denote $\ell^p Z_k(X) : = \ker \partial_k$ and $\ell^p B_k(X) : = \Im \partial_{k+1}$.
The \textit{$k$-th $\ell^p$-homology group} of $X$ (resp. \textit{$k$-th reduced $\ell^p$-homology group}) is the space:
\begin{equation*}
    \ell^p H_k(X) : = \ell^p Z_k(X) / \ell^p B_k(X) \quad (\textrm{resp. } \ell^p \overline{H_k}(X) : = \ell^p Z_k(X) /\overline{ \ell^p B_k(X) }).
\end{equation*}
Denote $\ell^p Z^k(X) : = \ker d_k$ and $\ell^p B^k(X) : = \Im d_{k-1}$. The \textit{$k$-th $\ell^p$-cohomology group} of $X$ (resp. \textit{$k$-th reduced $\ell^p$-cohomology group}) is the space:
\begin{equation*}
    \ell^p H^k(X) = \ell^p Z^k(X) / \ell^p B^k(X) \quad (\textrm{resp. } \ell^p \overline{H^k}(X) = \ell^p Z^k(X) /\overline{ \ell^p B^k(X)}).
\end{equation*}
The closures are considered with respect to the topology of the $\ell^p$-norm. The quotient spaces are endowed with the corresponding quotient topology. Reduced homology and cohomology groups are thus Banach spaces. The reduced versions are the greatest Hausdorff quotients of their non-reduced versions, thus a non-reduced homology or cohomology group is never Hausdorff unless it is equal to its corresponding reduced version.
\end{defn}

If $X$ is a finite simplicial complex, then these spaces correspond to classical homology and cohomology groups $H_k(X)$ and $H^k(X)$.
Thus we obtain new invariants only when the complex $X$ is non-compact. 

We will be mostly interested in the Davis realization of buildings. Since these are always contractible, all of these spaces are trivial for $k > 0$ if the building has finite thickness and its Weyl group is finite.

We review some elementary but important facts on $\ell^p$-homology and cohomology. First we have a duality result relating $\ell^p$-homology groups and $\ell^p$-cohomology groups but only for the reduced versions.

\begin{prop} \cite[1.1]{bourdon-survol}
For $p, r \in ] 1, + \infty[ $ such that $p^{-1} + r^{-1} = 1$ and $k \in \na$, the space $\ell^p \overline{H^k}(X)$ is canonically isomorphic to the dual of $\ell^r \overline{H_k}(X)$. Similarly, $\ell^r \overline{H_k}(X)$ is canonically isomorphic to the dual of $\ell^p \overline{H^k}(X)$.
\end{prop}

Second, $\ell^p$-cohomology is invariant by quasi-isometries. Before stating the result, we introduce the following condition. A complex of bounded geometry $(X, d)$ is said to be \textit{uniformly $\re$-acyclic} if $X$ is $\re$-acyclic and if there is a function $\Phi : \re_+ \to \re_+$ such that every cycle supported in a ball $B(x,r)$ is the boundary of a chain supported in $B(x, \Phi(r))$. The following result is due to M. Gromov \cite{gromov}, see \cite{bourdon-pajot} for a complete proof.

\begin{thm}\label{QIinvariance}
Let $X$ and $Y$ be two uniformly $\re$-acyclic complexes of bounded geometry. If $X$ and $Y$ are quasi-isometric, then the spaces $\ell^p H_k(X)$ and $\ell^p H_k(Y)$ are topologically isomorphic and the spaces $\ell^p H^k(X)$ and $\ell^p H^k(Y)$ are topologically isomorphic. This is also true for reduced homology and cohomology groups.
\end{thm}

\begin{rem}
The original statement needs "uniformly contractible" as a hypothesis, but by inspecting the proof in \cite{bourdon-pajot} we see that uniformly $\re$-acyclic is enough. We need to loosen this hypothesis since we will use the Bestvina complex of a Coxeter system, which is constructed a priori only as an $\re$-acyclic complex. It can be made contractible, but it may lose the property that its dimension computes a cohomological dimension when the latter is equal to 2 because of the Eilenberg-Ganea problem \cite[p. 154]{davisbook}, \cite{bestvina-cohom-dim}.
\end{rem}

Third, one can define continuous group cohomology for continuous representations of locally compact second countable (lcsc) groups \cite[Chapter IX]{borel-wallach}. In the case of a lcsc $G$ acting on some uniformly contractible simplicial complex $X$, simplicial $\ell^p$-cohomology of $X$ is related to continuous group cohomology of the regular representation of $G$ acting on $L^p(G)$.

\begin{prop}\label{simplicial cohom = group cohom} \cite[3.2]{bourdon-remy-vanishings}
    Let $G$ be a lcsc group endowed with a left Haar measure $\mu$ and $L^p(G) := L^p(G, \mu)$. Let $G$ act by simplcial automorphisms, properly and cocompactly on some uniformly contractible simplicial complex $X$ of bounded geometry. Denote by $H^*_{\mathrm{ct}}(G, L^p(G))$ the continuous group cohomology of $G$ with values in the right regular representation on $L^p(G)$. Then for every $p >1$, there are topological isomorphisms:
    \begin{equation*}
        H^*_{\mathrm{ct}}(G, L^p(G)) = \ell^p H^*(X).
    \end{equation*}
\end{prop}

Lastly, we have the following characterization of the first $\ell^p$-cohomology group of simply connected simplicial complexes, obtained by integration over 1-cycles. We will use this only in Section 5.

\begin{prop} \cite[Section 1.4]{bourdon-survol} Let $X$ be a simply connected simplicial complex of bounded geometry. Then for all $p>1$ we have canonical isomorphisms:
    \begin{equation*}
    \ell^pH^1(X) = \{ f: X^{(0)} \to \re, df \in \ell^p(X^{(1)}) \} / \ell^p(X^{(0)}) + \re.
\end{equation*}
\end{prop}

The Cayley graph $X$ of a finitely generated group $(\G, S)$ is not simply connected unless $(\G, S)$ is virtually free, so this proposition does not directly apply to Cayley graphs. Nevertheless, we can apply it to its Cayley complex $Y$ (i.e. the universal cover of its presentation complex), and the right hand side depends only on the 1-skeleton of $Y$, which is equal to $X$. We can thus define:
\begin{equation*}
    \ell^p H^1 (\G) : = \ell^p H^1(Y) =  \{ f: \G \to \re, df \in \ell^p(X^{(1)}) \} / \ell^p(\G) + \re.
\end{equation*}

\subsection{Chamber systems and geometric realizations}

We review some basic combinatorial building theory following \cite{ronan}. Then we define geometric realizations of buildings (following \cite[Chapter 12]{abramenko-brown} or \cite{davisbook}) and prove an $\re$-acyclicity criterion which applies to both the Davis realization and the Bestvina realization of buildings.

A \textit{chamber system} over a set $S$ is a set $\C$ together with a family of equivalence relations on $\C$ indexed by $S$. The elements of $\C$ are called \textit{chambers}. Two chambers are \textit{$s$-equivalent} if they are equivalent under the relation corresponding to some $s \in S$ and \textit{$s$-adjacent} if they are $s$-equivalent and not equal.

A \textit{gallery} $\g$ in $\C$ is a finite sequence of chambers $(c_0, \ldots, c_k)$ such that for every $1 \leq i \leq n$, there is some $s_i \in S$ such that $c_{i-1}$ is $s_i$-adjacent to $c_i$. The gallery $\g$ has \textit{type} $s = (s_1, \ldots, s_k)$ if $c_{i-1}$ is $s_i$-adjacent to $c_i$ for all $1 \leq i \leq n$. If $\g$ has type $(s_1, \ldots, s_k)$ where each $s_j$ belongs to a subset $T \subseteq S$, then we say $\g$ is a \textit{$T$-gallery}. A chamber system is \textit{connected} (or \textit{$T$-connected}) if any two chambers can be joined by a gallery (or $T$-gallery). The $T$-connected components of a chamber system $\C$ are called its \textit{$T$-residues}.

\begin{defn}
Let $(W,S)$ be a Coxeter system. A \textit{(combinatorial) building of type $(W,S)$} is a chamber system $\C$ over $S$ such that: \\
$(i)$ for all $s\in S$, each $s$-equivalence class contains at least two chambers and, \\
$(ii)$ there exists a \textit{$W$-valued distance function} $d_W : \C \times \C \to W$, that is, a map satisfying that: if $w$ is a reduced word in $(W,S)$, then two chambers $c$ and $c'$ can be joined by a gallery with type $(s_1, \ldots, s_k)$ if and only if $d_W(c,c') = s_1 \ldots s_k$. \\
The group $W$ is called the \textit{Weyl group} of the building $\C$.
\end{defn}

\begin{exmp}
Let $(W,S)$ be a Coxeter system. Then the chamber system $\C = W$, endowed with the relations $w \sim_s  w'$ when $w = sw'$ for each $s \in S$ and the $W$-valued distance $d_W : W\times W \to W$ defined by $d_W(w, w') = w^{-1}w'$, is a building of type $(W,S)$ called the \textit{abstract Coxeter complex of $W$}.
\end{exmp}

A building $\C$ of type $(W,S)$ \textit{has finite thickness} if for all $s \in S$, each $s$-equivalence class is finite. In this case, we say that $\C$ is \textit{regular} if for each $s \in S$, each $s$-equivalence class has the same number of elements $q_s + 1$. In this case we denote by $\mathbf{q}$ the vector containing the $q_s$'s as coordinates. We say that $\C$ \textit{has constant thickness} if all the $q_s$'s have the same value $q$. In the next sections our buildings will always be regular.

We now define geometric realizations of combinatorial buildings in a relatively general setting. This is because in what follows we will be interested in two realizations of buildings: the well-known Davis realization and the less known Bestvina realization. A \textit{mirror structure} on a simplicial complex $K$ consists of an index set $S$ and a family of subcomplexes $(K_s)_{s\in S}$. The subspaces $K_s$ are the \textit{mirrors} of $K$. In this case, we say that $K$ is a \textit{mirrored simplicial complex over $S$}. We set $K_\emptyset = K$ and for any nonempty subset $T\subseteq S$,
\begin{equation*}
    K_T = \bigcap_{t \in T} K_t.
\end{equation*}
For $x \in K$, we set $S(x) : = \{ s \in S, x \in K_s \} = \max \{ T \subseteq S, x \in K_T \}$.

\begin{defn} Let $\C$ be a (combinatorial) building of type $(W,S)$ and $K$ a mirrored simplicial complex over $S$. The \textit{$K$-realization} of $\C$ is the space:
\begin{equation*}
    X_K = (\C \times K) / \sim,
\end{equation*}
where $[(c,x)] \sim [(c',x')]$ if and only if $x = x'$ and $d_W(c, c') \in W_{S(x)}$. For simplicity, a simplex $[(c, \s)]$ in $X_K$ may be written $c. \s$ or $c \sigma$.
\end{defn}

Let $\mathcal{C}$ be a regular building of type $(W,S)$ and finite thickness. If $(K, d)$ is a geodesic metric space, we can extend $d$ to $X_K$ by declaring that all translates of $K$ are isometric and then defining a piecewise length metric on $X_K$. If moreover each $K_s$ is a proper metric space, then $(X_K, d)$ is a geodesic metric space \cite[ Corollary 12.28]{abramenko-brown}. The space $(X_K, d)$ has bounded geometry if and only if the mirror structure satisfies that for every subset $T \subseteq S$ generating an infinite subgroup $W_T$, we have $K_T = \emptyset$.

\begin{exmp} \label{simplicial realization}
Let $\C$ be a building of type $(W, S)$ and finite thickness, where $(W,S)$ is an irreducible Coxeter system.
We can choose $K = \D^{|S|-1}$ to be the standard simplex of dimension $|S|-1$. We endow $\D^{|S|-1}$ with a mirror structure by letting $\{ K_s, s \in S\} $ be the codimension 1 faces of $\D^{|S|-1}$. The space $X_\D = (\C \times K) / \sim $ is called the \textit{simplicial realization} or the \textit{Tits realization} of $\C$. This realization is locally finite if and only if every proper parabolic subgroup of $(W,S)$ is finite, and this happens exactly when $(W,S)$ is finite, affine or compact hyperbolic (in the sense of \cite[p.133, exercice 14]{bourbaki}).
\end{exmp}

The realization given in this example is rarely locally finite. This is why in Section \ref{Section Davis chamber} we will introduce the Davis chamber in order to obtain locally finite realizations for arbitrary finite thickness combinatorial buildings. 

Recall that an orientation on a simplex corresponds to a total order on its vertices. If one is given orientations on simplices of $K$, one has a natural orientation on all simplices of $X_K$ by declaring that all translates $c. \s$ of a simplex $\s$ in $K$ by $c \in \C$ have the same orientation. Once simplices have an orientation, we can talk of homology for geometric realizations of buildings.

We now give a criterion for the $K$-realization of a building to be uniformly $\re$-acyclic, starting from data on the mirrored complex $K$. When $K$ is the Davis chamber, the Davis realization $X$ is known to carry a $CAT(0)$ metric, so it is uniformly contractible. The main motivation to state this result here is to apply it to the Bestvina complex, which is not necessarily contractible. This is the same proof as \cite[8.2.8]{davisbook} but for buildings.  First, set $K^{\emptyset} = \emptyset$ and for any nonempty subset $T\subseteq S$,
\begin{equation*}
    K^T = \bigcup_{t \in T} K_t.
\end{equation*}

\begin{prop}\label{UniformlyAcyclic}
Let $\C$ be a building of type $(W,S)$ and finite thickness and let $X$ be its $K$-realization. Suppose that $X$ is locally finite and that $K$ is a geodesic metric space. If $K$ and $K_T$ are $\re$-acyclic for every $T\subseteq S$ generating a finite parabolic subgroup of $W$, then $X$ is uniformly $\re$-acyclic.
\end{prop}

\begin{proof}
Since $\C$ has finite thickness, the set of chambers of $\C$ is countable. We can order them $c_0, c_1, \ldots, c_n, \ldots$ so that $c : = c_0$ and $l(c_k) \leq l(c_{k+1})$, where $l(x) = l_S(d_W(c, x))$ is the length of the $W$-distance with respect to $c$. We set:
\begin{equation*}
    P_n = \bigcup_{i = 0}^n c_i K \subseteq X.
\end{equation*}
We will show that $P_n$ is $\re$-acyclic by induction. Since $P_{n} = P_{n-1} \cup c_n K$ and $c_n K$ is $\re$-acyclic by hypothesis, it is enough to show that the intersection $P_{n-1} \cap c_n K$ is $\re$-acyclic. 

The main step is showing the building analogue of \cite[8.1.1]{davisbook}. For $w \in W$, we set $\textrm{In}(w) = \{ s \in S, l(ws) < l(w) \}$. For all $i\geq 0$, set $w_i = d_W(c, c_i)$. Our goal is to show the equality $P_{n-1} \cap c_n K = c_n K^{\textrm{In}(w_n)}$. Let $\S$ be an apartment containing both $c$ and $c_n$. The inclusion $P_{n-1} \cap c_n K \supset c_n K^{\textrm{In}(w_n)}$ follows from \cite[8.1.1]{davisbook} applied to the apartment $\S$. For the converse inclusion, we show that for all $i < n$, we have $c_iK \cap c_nK \subset c_n K^{\textrm{In}(w_n)}$. Indeed, let $\rho : X \to \S$ be a retraction onto the apartment $\S$, which we identify with the $K$-realization of the abstract Coxeter complex of $(W,S)$. Since $c_n K$ is contained in $\S$, the set $c_iK \cap c_nK$ is fixed under $\rho$. Thus:
\begin{align*}
      c_iK \cap c_nK  & = \rho(c_iK \cap c_nK) \\
    & \subset \rho(c_iK) \cap \rho(c_nK) \\
    & = w_i K \cap w_n K.
\end{align*}
Now, $l(c_i) \leq l(c_n)$ means that $l(w_i) \leq l(w_n)$ and the proof of \cite[8.1.1]{davisbook} shows that $w_i K \cap w_n K \subset w_n K^{\textrm{In}(w_n)} =  c_n K^{\textrm{In}(w_n)}$. This shows the converse inclusion and thus we have that $P_{n-1} \cap c_n K = c_n K^{\textrm{In}(w_n)}$.

Since $W_{\textrm{In}(w)}$ is a finite parabolic subgroup of $W$ \cite[4.7.2]{davisbook}, our hypothesis implies that $c_n K^{\textrm{In}(w_n)}$ is $\re$-acyclic. Indeed, if $W_T$ is a finite parabolic subgroup, then $K^T$ is the union of $\re$-acyclic spaces (the $K_t$'s for $t \in T$) whose intersections are all $\re$-acyclic (the $K_{T'}$'s for $T' \subseteq T$) by hypothesis. 

Thus Mayer Vietoris says that $H_*(P_n, \re) = H_*(P_{n-1}, \re)$ and by induction $H_*(P_n, \re) = H_*(cK, \re) = 0$, so $P_n$ is $\re$-acyclic for all $n$. Since $X$ is the increasing union of the $P_n$, the space $X$ is $\re$-acyclic. 

Now suppose that $K$ is a geodesic metric space and endow $X$ with the corresponding piecewise length metric. The space $X$ is uniformly $\re$-acyclic since any cycle supported in a ball $B(x,r)$ is contained, up to changing the chamber $c$ so that $x \in cK$, in a certain $P_n$, which is $\re$-acyclic. This $P_n$ can be chosen so that $P_n \subseteq B(x ,C r + D)$ where $C, D > 0$ are two constants. Since every chamber is an isometric copy of $K$, these constants may depend on $r$ but are independent on the center $x \in X$.
\end{proof}

All the simplicial complexes we are interested in this article are locally finite. For a locally finite simplicial complex $X$ and $k \in \na$, we define locally finite $k$-chains on $X_K$ as (possibly infinite) formal sums of $k$-simplices in $X_K$ with coefficients in $\re$. Denote by $C_k^{\mathrm{lf}}(X_K; \re)$ their space and define the boundary operators as usual. Notice that the boundary operators only act on the second variable, so if $c. \s$ is a $k$-simplex in $X_K$, we have:
\begin{equation*}
    \partial_k (c. \s) = c.( \partial_k \s).
\end{equation*}
The \textit{locally finite homology $H_k^{\mathrm{lf}}(X_K; \re)$ of $X_K$} is the homology of the resulting chain complex. 

\subsection{Growth function and growth rate}

We come back to a slightly more general setting to define the growth function of a Coxeter system $(W,S)$. Its corresponding growth rate is a numerical invariant that we will compare to other numerical invariants.

For $w \in W$ denote by $\overline{w}$ the conjugacy class of $w$. Let $\overline{S} = \{ \overline{s}, s \in S \}$ and let $j : \overline{S} \to \{ 1, \ldots, N \}$ be a bijection. For each $\overline{s} \in \overline{S}$ introduce a variable $t_{j(\overline{s})} \in \re$. Let $w \in W$ and choose a reduced expression $s_1 \ldots s_k$ of $w$. The monomial
\begin{equation*}
    t_w = t_{j(\overline{s_1})}\ldots t_{j(\overline{s_k})}
\end{equation*}
is independent of the reduced expression $s_1 \ldots s_k$ of $w$. This follows from their characterization as minimal galleries from $1$ to $w$ on the abstract Coxeter complex. Indeed, a minimal gallery from $1$ to $w$ crosses the walls separating $1$ and $w$ exactly once, and only the order in which these walls are crossed depends on the minimal gallery.
\begin{defn}
The \textit{growth function} $W(\mathbf{t})$ of $(W,S)$ is the series:
\begin{equation*}
    W(\mathbf{t}) = W(t_1, \ldots, t_N) = \sum_{w\in W} t_w.
\end{equation*}
\end{defn}

For $x \in \re$, we denote by $\mathbf{t}^x$ the vector $(t_1^x, \ldots, t_N^x)$. For $\a \in \re$, the inequality $\mathbf{t}>\a$ (resp. $\mathbf{t}\geq \a$) means $t_i> \a$ (resp. $\mathbf{t_i}\geq \a$) for all $i = 1, \ldots, N$. For a fixed $\mathbf{t} \in \re^{N}$, we study the convergence of the series $W(\mathbf{t}^{-x})$ in function of $x$.

\begin{prop}\label{GrowthConvergence}
Fix $\mathbf{t} \in \re^N$ such that $\mathbf{t}>0$. The series $W(\mathbf{t}^{-x})$ converges if $x > e_\mathbf{t}(W)$ and diverges if $x < e_\mathbf{t}(W)$, where
\begin{equation*}
    e_\mathbf{t}(W) = \limsup_n \frac{1}{n} \log | \{ w \in W, t_w \leq e^n \} |.
\end{equation*}
\end{prop}

\begin{proof}
Set $Q_n(\mathbf{t}) = \{ w \in W, t_w \leq e^n \} $. The power series $\sum_n |Q_n(\mathbf{t})| e^{-xn}$ in the variable $x$ converges for $x > e_\mathbf{t}(W)$ and diverges for $x < e_\mathbf{t}(W)$. 
On the other hand we have:
\begin{equation*}
    \sum_n |Q_n(\mathbf{t})|e^{-xn} = \sum_{w\in W} \sum_{\substack{n \in \na \\ e^n \geq t_w}} e^{-xn}
\end{equation*}
and the term $\sum_{e^n \geq t_w} e^{-xn}$ is equivalent to $(t_w)^{-x} \frac{1}{1 - e^{-x}}$, so the previous series converges if and only if $W(\mathbf{t}^{-x})$ converges.
\end{proof}

\begin{rem}
If the $t_i$'s are all equal to some real number $t$, then the series $W(t)$ depends only on one variable and may be expressed as
\begin{equation*}
    W(t) = \sum_{k \geq 0} c_k(W) t^k,
\end{equation*}
where $c_k(W) = | \{ w \in W, l(w) = k \} |$. In this case the radius of convergence of the series $W(t^{-x})$ is given by
\begin{equation*}
    e_t(W) = \frac{1}{\log t} \limsup_n \frac{1}{n} \log c_n(W).
\end{equation*}
The number $e(W) = \limsup_n \frac{1}{n} \log c_n(W)$ is also called the \textit{exponential growth rate} of $(W,S)$. Since $W$ is finitely generated, $e(W)$ is always finite.
\end{rem}

In practice, if we consider a regular building of finite thickness $\mathbf{q} + 1$, the vector $\mathbf{ t}$ will be $\mathbf{q}$. The numbers $e_\mathbf{q}(W)$ correspond to exponential growth rates of balls in the building for a combinatorial metric taking into account the thickness $\mathbf{q} + 1$. 

If $\mathbf{t} = (t_1, \ldots, t_N)$ and $\mathbf{t}' = (t_1', \ldots, t_N')$ are two vectors with $1\leq t_i < t_i' $ for all $i$ and $e_{\mathbf{t}'}(W) > 0$, then $e_\mathbf{t}(W) > e_{\mathbf{t}'}(W)$. Indeed, there exists $\a > 1$ such that $t_i^\a < t_i'$ for all $i$. Then:
\begin{equation*}
    Q_n(\mathbf{t}') \subseteq Q_n(\mathbf{t}^\a) \subseteq Q_n(\mathbf{t})  
\end{equation*}
Since $q_w^\a \leq e^n \iff q_w \leq e^{n/\a}$, we have that $ Q_n(\mathbf{t}^\a) =  Q_{n/\a}(\mathbf{t})$. Thus:
\begin{equation*}
    e_{\mathbf{t}'}(W) \leq e_{\mathbf{t}^\a}(W) = \frac{1}{\a} e_\mathbf{t}(W) < e_\mathbf{t}(W).
\end{equation*}

This implies for instance that, if $\mathbf{t} > 1$ then we have:
\begin{equation*}
    \frac{1}{\log t_\mathrm{max}} e(W) \leq e_\mathbf{t}(W) \leq \frac{1}{\log t_\mathrm{min}} e(W),
\end{equation*}
where $t_\mathrm{min} = \min_{1\leq i \leq N} t_i$ and $t_\mathrm{max} = \max_{1\leq i \leq N} t_i$.
Hence when $\mathbf{t} > 1$ the growth rate $e_\mathbf{t}(W)$ is finite and we have $e_\mathbf{t}(W) = 0$ if and only if $e(W)= 0$.

\subsection{Retractions}

We now introduce our main tool: retractions. These are maps first defined at a combinatorial level, which allow to extend $W$-invariant metrics from a Coxeter complex to the whole building. We will use the maps they induce at a homological level.

Let $\C$ be a regular building of type $(W,S)$ and choose a chamber $c \in \C$. The map $\rho_c : \C \to W$ defined by $\rho_c(c') = d_W(c, c')$ is the \textit{retraction onto the abstract Coxeter complex of $W$ centered in $c$}. Retractions are naturally defined on $K$-realizations of $\C$ by acting on the first variable. The set
\begin{equation*}
    S(c, w) = \rho_c^{-1}(w) = \{ c' \in \C, d_W(c, c') = w \}
\end{equation*}
is called the \textit{$W$-sphere of radius $w$ centered in $c$}. 

Since $\C$ is assumed to be regular, denote by $q_s + 1$ the cardinal of $s$-equivalence classes for each $s \in S$. If $\overline{s}_i = \overline{s}_j$ then $q_{s_i} = q_{s_j}$ \cite[18.1.16]{davisbook}. Denote by $\mathbf{q}$ the vector containing the $q_s$'s. For $w \in W$, let $w = s_1 \ldots s_k$ be a reduced expression for $w$ as a word in $S$. By the same argument as in the previous section, the number 
\begin{equation*}
    q_w = q_{s_1} \ldots q_{s_k}
\end{equation*}
is independent of the decomposition $w = s_1 \ldots s_k$.
If $\C$ has constant thickness equal to $q+1$, we have $q_w = q^{l(w)}$.

\begin{lem}\cite[18.1.17]{davisbook}
For all $w \in W$, we have $ | S(c, w) | = q_w$.
\end{lem}

Let $(W,S)$ be a Coxeter system. Let $X$ be the $K$-realization of a building $\C$ of type $(W,S)$ and $\S$ the $K$-realization of the Coxeter complex of $(W,S)$. Choose an orientation on simplices of $K$ so that all simplices in $X$ and $\S$ have an orientation.

A retraction $\rho : X \to \S$ centered on a chamber defines two maps between chains on $X$ and chains on $\S$. First, we can consider the \textit{pushforward of $\rho$}, retracting chains on $X$ onto chains on $\S$:
\begin{align*}
    \rho_*& : C_k^{\mathrm{lf}}(X; \re) \to C_k^{\mathrm{lf}}(\S; \re) \\
    & (c, \s) \mapsto (\rho(c) , \s)
\end{align*}

Second, we have the \textit{pullback of $\rho$}, extending chains on $\S$ into chains on $X$:
\begin{align*}
    \rho^* &:  C_k^{\mathrm{lf}}(\S; \re) \to C_k^{\mathrm{lf}}(X ; \re) \\
    & (w, \s) \mapsto \sum_{c' \in \rho^{-1}(w)} \frac{1}{q_w} (c' , \s)
\end{align*}

Notice that $\rho_* \circ \rho^* = \Id_{C_k^{\mathrm{lf}}(\S)}$ and that $\rho^* \circ \rho_*$ corresponds to averaging chains over $W$-spheres in $X$.

\begin{prop}
The maps $\rho_*$ and $\rho^*$ commute with boundary operators. \\
The map in homology induced by $\rho^* : H_k^{\mathrm{lf}}(\S; \re) \to H_k^{\mathrm{lf}}(X; \re) $ is injective. \\ 
The map in homology induced by $\rho_* : H_k^{\mathrm{lf}}(X; \re) \to H_k^{\mathrm{lf}}(\S; \re)$ is surjective.
\end{prop}

\begin{proof}
The maps $\rho_*$ and $\rho^*$ act only on the first variable of $X_K = (\C \times K) / \sim $ while the boundary operators act only on the second variable, hence $\rho_*$ and $\rho^*$ commute with boundary operators. Injectivity and surjectivity of the maps induced in homology follow from the identity $\rho_* \circ \rho^* = \Id_{C_k^{\mathrm{lf}}(\S; \re)}$.
\end{proof}

\subsection{Davis complex and Davis chamber}\label{Section Davis chamber}

We now define the Davis realization $X_D$ of a building. The main reason to introduce this realization is that it can always be endowed with a $CAT(0)$-metric \cite{daviscat0}. We give two different (but equivalent) constructions of $X_D$. We first follow \cite{dymarajanuszkiewicz}.

Let $X$ be the simplicial realization of a non-spherical regular building $\C$ of type $(W,S)$ and finite thickness. The space $X$ does not enjoy good topological properties as usually it is not locally finite (see Example \ref{simplicial realization}). The following construction gives a simplicial complex $X_D$ that is topologically similar to $X$ but that enjoys better properties. 

Recall that the link of a simplex $\s$ in a simplicial complex is the set: \begin{equation*}
    \Lk(\s) = \{ \t \textrm{ simplex } , \, \s \textrm{ and } \t \textrm{ are disjoint faces of the same maximal simplex} \}.
\end{equation*}

\begin{defn}\label{Davis complex DJ}
Let $X'$ be the first barycentric subdivision of $X$. The \textit{Davis complex} $X_D$ of $X$ is the subcomplex of $X'$ generated by the barycenters of simplices of $X$ with compact links. 
\end{defn}

The space $X_D$ is a deformation retract of $X$ \cite[1.4]{dymarajanuszkiewicz} and is locally finite. Since $X_D$ can be endowed with a $CAT(0)$-metric, both $X$ and $X_D$ are contractible. Note that with this definition $X_D$ is a simplicial complex, but it is not necessarily purely dimensional, that is, its maximal simplices may not have the same dimension.

The intersections $D = \D \cap X_D$ are isomorphic for any simplicial chamber $\D$ in $X$ (because they are all translates of a given chamber by elements of the Weyl group for some apartment). We call such an intersection $D$ a \textit{Davis chamber} of $X_D$.

As our notation suggests, the space $X_D$ can also be constructed from the combinatorial building $\C$ and a Davis chamber $D$, after seeing it can be endowed with the structure of a mirrored space \cite[Chapter 18]{davisbook}.

\begin{defn}
The \textit{Davis chamber $D = D(W,S)$ of $(W,S)$} is the (geometric realization of the) barycentric subdivision of the poset of all finite parabolic subgroups $W_T = \langle s\in T \rangle$ for $T\subseteq S$ in $W$ ordered by inclusion. For $s \in S$, define $D_s$ to be the subcomplex of $D$ corresponding to the poset of all finite parabolic subgroups of $W$ containing $W_{ \{s \} }$.
\end{defn}

Thus $D$ is a mirrored space over $S$ and both definitions of $X_D$ agree:
\begin{equation*}
    X_D = (\C \times D) / \sim.
\end{equation*}

Notice that for $T \subseteq S$ generating a finite subgroup $W_T$, the subcomplex $D_T$ is the cone on the poset of all finite parabolic subgroups of $W$ strictly containing $W_{ T }$, thus $D_T$ is contractible.

We will now define an orientation on the simplices of $D$. All simplices in $\D^{|S|-1}$ have the natural orientation given by $\re^{|S|-1}$. Any simplex of its barycentric subdivision inherits a natural orientation. Now we just consider the restriction of this orientation to $D$. This determines an orientation for every simplex in $X_D$.

\section{Exponent of top $\ell^p$-homology for buildings of type $PM$}

The first $\ell^p$-cohomology group increases as $p$ grows to infinity, thus there exists a critical $p$ for which it starts to be nonzero. This is a quasi-isometry invariant first introduced by Pansu \cite{pansu89}. Top degree $\ell^p$-homology has the same behaviour as $p$ grows, so there is again a critical exponent for which this space starts to be nonzero. We show that in the particular case of buildings of type PM, this critical exponent can be computed in terms of the Weyl group and the thickness.

\subsection{The nerve and Coxeter groups of type PM}

In this paragraph we define Coxeter groups of type PM. These are Coxeter groups whose nerve satisfies some natural topological conditions. They can be seen as Coxeter systems whose corresponding Davis complex is close to an orientable manifold.
We first introduce the nerve of a Coxeter system. 

\begin{defn}
The \textit{nerve} $L = L(W,S)$ of a Coxeter system $(W,S)$ is the poset of all nonempty subsets $T \subseteq S$ such that the parabolic subgroup $W_T$ of $W$ is finite.
\end{defn}

Notice that the Davis chamber $D$ of $(W,S)$ is the cone on the barycentric subdivision of $L$, with apex the vertex corresponding to the empty set. If we view a Davis chamber $D \subseteq \D$ inside a simplicial chamber $\D$, we may also see the barycentric subdivision $L'$ as the intersection $D \cap \partial \D$.

Now we introduce some natural topological conditions on simplicial complexes. A Coxeter group will be of type PM if its nerve satisfies these conditions.

\begin{defn}
Let $Y$ be a locally finite purely $n$-dimensional simplicial complex. The space $Y$ is a \textit{pseudomanifold} if every simplex of codimension 1 is contained in exactly two maximal simplices. We say that $Y$ is \textit{orientable} if there exists an orientation on the simplices of $Y$ so that the sum of all of its maximal simplices is an $n$-cycle in $H_n^{\mathrm{lf}}(Y)$. We say that $Y$ is \textit{gallery connected} if any two maximal simplices in $Y$ can be joined by a finite sequence of maximal simplices such that any two consecutive maximal simplices share a face of codimension 1.
\end{defn}

The fundamental homological property that simplicial complexes with the above conditions satisfy is the following.

\begin{prop}\label{homology-pseudomanifold-is-one-dim}
    Let $Y$ be an orientable, gallery connected pseudomanifold of dimension $n$. Then $H_n^{\mathrm{lf}}(Y; \re)$ is one dimensional.
\end{prop}

\begin{proof}
    The space $H_n^{\mathrm{lf}}(Y; \re) = Z_n^{\mathrm{lf}}(Y; \re)$ is non-trivial because of our orientability assumption. Now let $\eta = \sum_\s a_\s \s \in Z_n^{\mathrm{lf}}(Y; \re)$ be such that we have $a_{\t} = 0$ for some $n$-simplex $\t$. Let $\s$ be any other $n$-simplex in $Y$. Since $Y$ is gallery connected, there exists a gallery of $n$-simplices $(\s_0, \s_1, \ldots, \s_d) $ joining $\t = \s_0$ to $\s =  \s_d$. Since $Y$ is a pseudomanifold, the chain condition on each codimension 1 face gives $a_{\s_{i-1}} + (-1)^{\epsilon}a_{\s_i} = 0$ for each $i = 1, \ldots, d$ (where $\epsilon$ depends on the orientations). Since $a_{\t} = 0$, this implies inductively that $a_{\s_i}$ for every $i$, so $a_\s = 0$ and hence $\eta = 0$. Any two locally finite $n$-cycles coincide on some simplex if we rescale one of them, hence the previous argument shows that they have to be proportional.
\end{proof}

\begin{defn}
A Coxeter system $(W,S)$ is \textit{of type PM} if its nerve $L$ is an orientable, gallery connected pseudomanifold. A building is \textit{of type PM} if its Weyl group is of type PM.
\end{defn}

For a Coxeter system $(W,S)$ of type $PM$, the Davis realization of the Coxeter complex $\S_D$ is a contractible orientable pseudomanifold.

\begin{rem}
Due to our orientations choice, the canonical $n$-cycle on $\S_D$ will not be the sum of all chambers, but the \textit{alternate} sum of all chambers 
\begin{equation*}
    \sum_{w \in W} (-1)^{l(w)} w.D.
\end{equation*}
\end{rem}

\begin{exmp}
Affine and compact hyperbolic Coxeter groups (in the sense of \cite[p.133, exercice 14]{bourbaki}) are of type PM. Indeed, in these cases the Davis chamber $D$ is just (the barycentric subdivision of) a simplex, and its nerve is a triangulation of a sphere.

Coxeter groups generated by reflections on codimension 1 faces of a right-angled polyhedron $P$, as defined in \cite[p. 614]{dymarajanuszkiewicz}, are of type PM (e.g. the tiling of the hyperbolic plane by compact regular pentagons). The nerve is simplicially isomorphic to the dual of the boundary of the polyhedron $P$, so it is again a triangulation of a sphere.

In \cite{januszkiewicz-swiatkowski}, Januszkiewicz and \'{S}wiątkowski exhibit an infinite family of Gromov-hyperbolic Coxeter groups of type $PM$ whose Davis complexes have unbounded dimension. However, they also show that if such a group satisfies Poincaré duality, the dimension of its Davis complex is at most 61.
\end{exmp}

\subsection{Computation of the critical exponent}

In this section we will suppose the Coxeter system $(W,S)$ is of type PM. Let $X_D$ be the Davis realization of a regular building $\C$ of type $(W,S)$ and thickness $\mathbf{q} + 1$. This proof is essentially the same as that of \cite{dymarafundclass} for $\ell^2$-cohomology.

\begin{prop}\label{Jensen} Let $X_D$ be the Davis realization of a regular building $\C$ and $n = \dim X_D$. Let $\rho : X_D \to \S_D$ be a retraction onto the abstract Coxeter complex.
Let $\eta \in C_n^{\mathrm{lf}}(X_D, \re)$. Then for all $p >1$ 
\begin{equation*}
    ||\rho^*  \rho_* (\eta) ||_p \leq ||\eta||_p.
\end{equation*}
\end{prop}

\begin{proof} Write $\eta = \sum_{c \in \C, \s \in D^{n}} a_{c,\s} c.\s$. The operator $\rho^* \rho_*$ averages over $W$-spheres, thus Jensen's inequality yields the result:
\begin{align*}
    ||\rho^* \rho_* (\eta) ||_p^p & = \sum_{w\in W} \sum_{c \in \rho^{-1}(w)} \sum_{\s \in D^{(n)}} \frac{1}{q_w^p} \, \Big| \sum_{c' \in \rho^{-1}(w)} a_{c', \s} \Big| ^p \\
    & \leq \sum_{w\in W} \sum_{c \in \rho^{-1}(w)} \sum_{\s \in D^{(n)}} \frac{1}{q_w} \, \sum_{c' \in \rho^{-1}(w)} |a_{c', \s}| ^p \\
    & =  \sum_{w\in W} \, \sum_{c' \in \rho^{-1}(w)} \sum_{\s \in D^{(n)}} |a_{c', \s}| ^p = ||\eta||_p^p.
\end{align*}
\end{proof}

\begin{thm}\label{criticalPM}
Let $(W,S)$ be a Coxeter system of type $PM$. Let $X_D$ be the Davis realization of a regular building of type $(W,S)$ and thickness $\mathbf{q} + 1$, with $\mathbf{q} \geq 2$. Let $n = \dim X_D$. Then we have:
\begin{align*}
    1 + e_\mathbf{q}(W) = \inf \{ p > 1 \, | \, \ell^p H_n(X_D) \neq \{0\} \}, \\
    1 + e_\mathbf{q}(W)^{-1} = \sup \{ p > 1 \, | \, \ell^p \overline{H^n}(X_D) \neq \{0 \} \}.
\end{align*}
\end{thm}

\begin{proof}
The Davis realization of the Coxeter complex $\S_D$ is an orientable pseudomanifold, so by Proposition \ref{homology-pseudomanifold-is-one-dim} the $n$-chain $\t = \sum_{w\in W} (-1)^{l(w)} w.D$ is the only non-trivial locally finite $n$-cycle on $\S_D$ up to a multiplicative constant. Thus $\rho^* (\t)$ is a non-trivial locally finite cycle on $X_D$. We compute its $\ell^p$-norm:
\begin{equation*}
    ||\rho^* (\t)||_p^p = \sum_{w \in W} \sum_{c \in \rho^{-1}(w)} \frac{1}{q_w^{p}} = \sum_{w \in W} q_w^{1-p} = W(\mathbf{q}^{1-p}).
\end{equation*}
By Proposition \ref{GrowthConvergence}, this converges for $p > 1 + e_\mathbf{q}(W)$, so that $\ell^p H_n(X_D) \neq \{0\}$ for $p > 1 + e_\mathbf{q}(W)$.

Now suppose that $\ell^p H_n(X_D) \neq \{0\}$ and let $\eta \in \ell^p H_n(X_D)$ be a nonzero $\ell^p$-cycle. We can choose to center the retraction $\rho$ on a chamber of $X_D$ where $\eta$ is nonzero. Because of this choice, $\rho_*(\eta)$ is a nonzero multiple of $\t$, so we can assume $\rho_*(\eta) = \t$. Then by Proposition \ref{Jensen}
\begin{equation*}
    ||\rho^*(\t) ||_p = ||\rho^* \rho_*(\eta)||_p \leq ||\eta||_p < \infty.
\end{equation*}
But $||\rho^*(\t) ||_p$ diverges when $p < 1 + e_\mathbf{q}(W)$. Thus we have $\ell^p H_n(X_D) = \{0\}$ for $p < 1 + e_\mathbf{q}(W)$.

Results on reduced cohomology follow since in top degree we have $\ell^p H_n(X_D) = \ell^p \overline{H_n}(X_D) = \ell^r \overline{H^n}(X_D)$ where $p^{-1} + r^{-1} = 1$. For such $p$ and $r$ the condition $1 < p \leq 1 + e_\mathbf{q}(W)$ is equivalent to $ 1 + e_\mathbf{q}(W)^{-1} \leq r < \infty$.
\end{proof}

Because of the thickness assumptions, the number $e_\mathbf{q}(W)$ is finite and $e_\mathbf{q}(W)= 0$ if and only if the exponential growth rate $e(W)$ of $(W,S)$ satisfies $e(W) = 0$. By the strong Tits' alternative for Coxeter groups \cite[17.2.1]{davisbook}, an infinite Coxeter group $W$ is affine if and only if $W$ has subexponential growth, which in turn is equivalent to $e_\mathbf{q}(W) = 0$. We obtain the following results for $\ell^p$-cohomology:
 
\begin{cor}\label{affineVSnon-affine} Let $X_D$ be the Davis realization of a regular building of type $(W,S)$ and thickness $\mathbf{q} + 1$, with $\mathbf{q} \geq 2$ and $\dim X_D = n$. \\
1. If the group $W$ is affine, $\ell^p \overline{H^n}(X_D) \neq \{0\}$ for all $p > 1$. In particular, any semisimple group $G$ of rank $n$ over a non-Archimedean local field satisfies $\overline{H_\mathrm{ct}^n}(G, L^p(G)) \neq \{0\}$ for all $p > 1$.  \\
2. If the Weyl group $(W,S)$ is of type PM and non-affine, then $0 < e_\mathbf{q}(W) < \infty$ so we have $ \ell^p \overline{H^n}(X_D) \neq \{0\}$ for all $p < 1 +e_\mathbf{q}(W)^{-1}$ and $ \ell^p \overline{H^n}(X_D) = \{0\}$ for all $p > 1 +e_\mathbf{q}(W)^{-1}$.
\end{cor}

This result and quasi-isometric invariance of reduced $\ell^p$-cohomology imply that the critical exponents $1 + e_\mathbf{q}(W)$ and $1 + e_\mathbf{q}(W)^{-1}$ (and hence $e_\mathbf{q}(W)$) are quasi-isometry invariants for regular buildings of type $PM$ with finite thickness.

\subsection{Application to cocompact Fuchsian buildings}

As an application, the previous theorem completes the computation of all possible vanishings of $\ell^p$-cohomology for cocompact Fuchsian buildings. A building $X$ of type $(W,S)$ is said to be \textit{cocompact Fuchsian} if $W$ is a Fuchsian Coxeter group that acts properly and cocompactly on the hyperbolic plane $\mathbb{H}^2$. In particular, the Davis realization of $(W,S)$ is a tiling of $\mathbb{H}^2$ and thus these buildings are of type $PM$.

First we recall what is known in degree 1. In \cite{bourdon-fuchsien}, Bourdon computes the conformal dimension (see Section 6) of these buildings by the formula $\Cdim(\partial X) = 1 + e_\mathbf{q}(W)^{-1}$. The proof of this result relies on constructing a visual metric on the boundary $\partial X$ with the right parameter and Hausdorff dimension. Equipped with this metric, $\partial X$ is Loewner and by \cite[Théorème 0.3]{bourdon-pajot}, we obtain that $\ell^p H^1(X) \neq 0$ if and only if $p > \Cdim(\partial X) = 1 + e_\mathbf{q}(W)^{-1}$. The vanishing part of this statement is specific to these buildings, while the non-vanishing part holds for any Gromov-hyperbolic complex with bounded geometry (see Proposition \ref{Poisson-transform}).  

In \cite{bourdon-lp-degre-superieur}, Bourdon obtained that for $p > \Cdim(\partial X)$ we have $\ell^p H^2(X) = 0$. Since $\Cdim(\partial X) = 1 + e_\mathbf{q}(W)^{-1}$, Theorem \ref{criticalPM} (in particular its non-vanishing statement) implies that Bourdon's vanishing is optimal in this case. We sum these results in the following statement.

\begin{cor}\label{applicationFuchsian}
Let $X$ be the Davis realization of a cocompact Fuchsian building of type $(W,S)$ and thickness $\mathbf{q} + 1$, with $\mathbf{q} \geq 2$. \\
$\bullet$ For $p < \Cdim(\partial X) =  1 + e_\mathbf{q}(W)^{-1}$, we have 
\begin{center}
    $\ell^p H^1(X) = \{0\}$ and $\ell^p \overline{H^2}(X) \neq \{0\}$. 
\end{center}
$\bullet$ For $p > \Cdim(\partial X) =  1 + e_\mathbf{q}(W)^{-1}$, we have
\begin{center}
    $\ell^p H^1(X) \neq \{0\}$ and $\ell^p H^2(X) = \{0\}$.
\end{center}
\end{cor}

\section{Non-vanishing of $\ell^p$-homology and virtual cohomological dimension}

The main result of this section is that for large $p$, the $\ell^p$-homology of a building with Weyl group $W$ does not vanish in degree equal to the virtual cohomological dimension $\vcd_\re(W)$ of $W$ over $\re$. This is a generalization of the non-vanishing assertion shown for buildings of type $PM$ in the previous section. We begin by giving a quick review on the notion of virtual cohomological dimension for Coxeter groups. Then we introduce the Bestvina chamber and the Bestvina realization of a building following \cite{bestvina-cohom-dim} to obtain this non-vanishing result.

\subsection{Virtual cohomological dimension}

First let $\G$ be any discrete group and $R$ be a PID. The \textit{cohomological dimension $\cd_R(\G)$ of $\G$ over $R$} is the least $n$ such that the trivial $R\G$-module $R$ has a projective resolution of length $n$:
\begin{equation*}
0 \to P_n \to \ldots \to P_1 \to P_0 \to R \to 0
\end{equation*}
(and is $\infty$ if there is no such integer). Usually we denote $\cd(\G) := \cd_\ze(\G)$. One can show that this number corresponds indeed to some sort of dimension in a more intuitive way:
\begin{equation*}
    \cd_R(\G) = \sup \{ n \, , \, H^n(\G ; M) \neq \{0\} \textrm{ for some } R\G \textrm{-module } M  \}.
\end{equation*}

The group $\G$ is said to be \textit{of type $FP_R$} if there exists a finite length projective resolution of $R$ by finitely generated projective $R\G$-modules. In this case, the cohomological dimension of $\G$ over $R$ can be computed as follows:
\begin{equation*}
    \cd_R(\G) = \sup \{ n \, , \, H^n(\G ; R \G) \neq \{0\}  \}.
\end{equation*}

If $\cd(\G) < \infty$, then $\G$ is torsion-free. A Coxeter group is never torsion-free, so cohomological dimension is not directly relevant for such groups since it is always infinite. A more interesting invariant for Coxeter groups can be obtained as follows.

First we have the following result by Serre \cite{serre-cohom-dim}.

\begin{thm}
Suppose $G$ is a torsion-free group and that $\G$ is a subgroup of finite index. Then $\cd_R \, G = \cd_R \, \G$. 
\end{thm}

Since the intersection of two subgroups of finite index is still of finite index, this theorem allows us to define virtual cohomological dimension as follows.  

\begin{defn}
Let $G$ be a group having a torsion-free subgroup of finite index. We define the \textit{virtual cohomological dimension $\vcd_R(G)$ of $G$ over $R$} as the cohomological dimension $\cd_R(\G)$ of any of its torsion-free subgroups of finite index.  
\end{defn}

 Let $(W,S)$ be a finitely generated Coxeter system. First, Tits showed that $W$ is a linear group in characteristic $0$ \cite[D.1.2]{davisbook}, so by Selberg's lemma \cite{selberg} it admits a torsion-free subgroup $\G$ of finite index. Second, such a group $\G$ is of type $FP_R$, for any PID $R$. This is because $\G$ acts freely and cocompactly on the Davis apartment $\S$ of $W$, which is contractible, so the spaces $C_i(\S; R)$ give the desired projective resolution. This implies:
 \begin{equation*}
     \cd_R(\G) = \sup \{ n \, , \, H^n(\G ; R \G) \neq \{0\}  \}.
 \end{equation*}
 Serre's theorem tells us that this integer does not depend on the choice of a torsion-free subgroup $\G$ of finite-index in $W$, but we can see this directly without invoking this result. Indeed, both $W$ and $\G$ act properly discontinuously and cocompactly on the Davis apartment $\S$, so the spaces $H^n(\G ; R \G)$ and $ H^n(W; R W)$ coincide with the compactly supported cohomology $H_c^n(\S ; R)$ of $\S$ \cite[F.2.2]{davisbook}. Thus the virtual cohomological dimension of $W$ can be computed as follows:
 \begin{equation*}
    \vcd_R(W) = \sup \{ n \, , \, H^n(W ; R W) = H_c^n(\S ; R) \neq \{0\}  \}.  
\end{equation*}

\subsection{Bestvina chamber}

The dimension of the Davis realization of a building gives an upper bound for the virtual cohomological dimension of its Weyl group, but in general there is no equality. The Bestvina chamber is a topological construction that associates to every Coxeter system $(W,S)$ a finite acyclic simplicial complex whose dimension coincides with $\vcd(W) = \vcd_\ze(W)$. One can consider variants where the complex is $\F$-acyclic for a given field $\F$, and in this case the dimension of the complex is $\vcd_\F(W)$.
Since we are interested in $\ell^p$-cohomology, we will choose $\F = \re$.

In this section we review \cite{bestvina-cohom-dim} for the construction of the Bestvina chamber and the computation of its dimension, since it will be useful for non-vanishing of $\ell^p$-homology.

First, Bestvina describes an inductive construction of the Davis chamber. Denote by $\mathcal{F}$ the poset of subsets $T \subseteq S$ such that $W_T = \langle T \rangle$ is a finite parabolic subgroup of $W$ ordered with respect to inclusion. For any maximal element $F \in \mathcal{F}$, define $P_F$ to be a point. Assuming that $P_{F'}$ has been constructed for every $F' \supset F$, define $P_F$ to be a cone on $\bigcup_{F' \supset F} P_{F'}$. The complex $P_\emptyset$ is the Davis chamber $D$ of $W$ and the complexes $P_{\{s\}}$ are mirrors of $D$.

Now we construct an $\re$-acyclic simplicial complex $B_\re = B_\re(W)$ as follows. 

\begin{defn}
    For any maximal element $F \in \mathcal{F}$, define $P_F$ to be a point. Assuming that $P_{F'}$ has been constructed for every $F' \supset F$, define $P_F$ to be a finite $\re$-acyclic simplicial complex containing $\bigcup_{F' \supset F} P_{F'}$ of the least possible dimension. We then define the \textit{Bestvina chamber} $B_\re$ to be $P_\emptyset$.

    The Bestvina chamber has a natural mirrored structure given by the subcomplexes $B_s = P_{ \{ s \} }, s \in S$. We can thus define the \textit{Bestvina realization $X_B$ of a building $\C$}, as the space:
\begin{equation*}
X_B = (\C \times B_\re )/ \sim.
\end{equation*}
\end{defn}

Most of the time, we have $\dim(P_F) = \dim(\bigcup_{F' \supset F} P_{F'}) + 1$ and $P_F$ is just a cone on $\bigcup_{F' \supset F} P_{F'}$, but in some situations (e.g. when there is a unique $F' \in \mathcal{F}$ such that $F' \supset F$ and $|F'| = |F| +1$) we have $\dim(P_F) = \dim(\bigcup_{F' \supset F} P_{F'})$. The following lemma gives a more general criterion to guarantee that $\dim(P_F) = \dim(\bigcup_{F' \supset F} P_F)$. We give the proof for completion and also because the hypotheses are slightly different from those in \cite{bestvina-cohom-dim}.

\begin{lem}\label{RacyclicPolyhedron}
Let $L$ be a compact $n$-dimensional simplicial complex with $H_n(L; \re) = 0$. Then $L$ embeds in a compact $\re$-acyclic $n$-dimensional simplicial complex as a subcomplex.
\end{lem}

\begin{proof}
The lemma is true for $ n = 1$. For $n \geq 2$ it is enough to show it for $L$ $(n-2)$-connected. Indeed, one can define successively $L_{-1} = L$ and $L_i$ to be $L_{i-1}$ with the cone on its $i$-skeleton attached, we then replace $L$ by $L_{n-2}$. This does not modify $H_n(L;\re)$.

Now suppose $L$ is $(n-2)$-connected. The only homology group of $L$ with real coefficients that may not vanish is $H_{n-1}(L;\re) = H_{n-1}(L) \otimes \re$. The Hurewicz homomorphism gives $ H_{n-1}(L) \simeq \pi_{n-1}^{ab}(L)$ (which is just $\pi_{n-1}(L)$ for $n \geq 3$). Thus the group $\pi_{n-1}^{ab}(L)$ is (abelian and) finitely generated, but not necessarily free. Choose a basis $[f_1], \ldots, [f_k]$ of the maximal free abelian subgroup of $\pi_{n-1}^{ab}(L)$. For each $i=1, \ldots, k$, we attach a cone to the image of the map $f_i: S^{n-1} \to L$. Call the resulting space $L'$. The homology group $H_{n-1}(L')$ is a torsion group and therefore $H_{n-1}(L';\re) = 0$. This procedure does not affect other homology groups, thus $L'$ is $\re$-acyclic.
\end{proof}

To show non-vanishing of $\ell^p$-homology of a building, we will use the cycle constructed by Bestvina in his proof of $\vcd_\re(W) = \dim B_\re$. We write its construction for completeness.

\begin{thm}\label{BestvinaCycle}
Let $d = \dim B$. Then the space $H_d^{\mathrm{lf}}(\S_B; \re)$ contains a non-zero locally-finite cycle with bounded coefficients.
\end{thm}

\begin{proof}
Denote by $\mathcal{F}$ the poset of subsets $T \subseteq S$ such that $W_T = \langle T \rangle$ is a finite subgroup of $W$ ordered by inclusion. Let $F_0$ be a maximal element in $\mathcal{F}$ such that $ \dim P_{F_0} =d$. Therefore $U = \bigcup_{F \in \mathcal{F}, F \supset F_0} P_F$ has dimension $d-1$ and $P_{F_0}$ is a cone on $U$. By Lemma \ref{RacyclicPolyhedron}, we have $H_{d-1}(U; \re) \neq 0$. Since $P_{F_0}$ is $\re$-acyclic, we also have $H_d(P_{F_0}, U ; \re) \neq 0$, so choose $\t_0$ a non-trivial cycle in this space. We see this element as a chain on $B_\re$ whose boundary lies in $U$ and such that its stabilizer in $W$ (with respect to its action on $\S_B = (W \times B_\re) / \sim$) is $W_{F_0}$.

Consider the set $S_0 = \bigcup_{F \in \mathcal{F}, F \supset F_0} F \setminus F_0 \subseteq S$ and its corresponding parabolic subgroup $W_{S_0}$ in $W$. Then the chain 
\begin{equation*}
    \t = \sum_{w \in W_{S_0}} (-1)^{l(w)} w.\t_0
\end{equation*}
is a nonzero locally finite cycle in $H_d^{\mathrm{lf}}(\S_B; \re)$. Indeed, it is a cycle since the only place where the boundary of $\t$ may be nonzero is in the translates of $U$ by $W_{S_0}$, and on each of these translates the boundary of $\t$ gives twice the same contribution but with opposite sign. The coefficients of $\t$ are bounded since it is obtained by placing copies of $\t_0$ in different chambers.
\end{proof}

\begin{cor}
We have $\vcd_\re(W) = \dim B_\re$.
\end{cor}

\begin{proof} Let $d = \dim B_\re$.
The group $W$ acts properly and cocompactly on the space $\S_B$, so we have $H^*(W, \re W) = H_c^*(\S_B, \re)$. Since $\vcd_\re(W)$ is the greatest $n$ such that $H^*(W, \re W)$ is non-zero, we have $\vcd_\re(W) \leq d$.

The natural pairing $H_d^{\mathrm{lf}}(\S_B; \re) \otimes H_c^d(\S_B; \re) \to \re$ is non-degenerate, so in particular $H^d(W, \re W) = H_c^d(\S_B ; \re) \neq 0$. Thus $\vcd_\re(W) = d$.
\end{proof}

Unlike the Davis realization, the Bestvina realization may not preserve the underlying combinatorial structure of the building (in the sense that we cannot reconstruct the chamber system from the Bestvina realization). For instance, the Bestvina complex of any finite Coxeter group is a point. In \cite{bestvina-cohom-dim}, the example of a $1$-dimensional Bestvina chamber in which the mirrors cover the whole chamber is treated. The Bestvina realization of the corresponding Coxeter group is a tree and the chambers are tripods which overlap on segments. Overlapping in top dimension is avoided exactly when $B = P_\emptyset$ is a cone on its mirrors. Even though the combinatorial structure of the building may not be respected by the Bestvina complex, the Davis and Bestvina realization of the same building are still quasi-isometric.

\begin{prop}\label{Quasi-isometry-XB-XD} Let $\mathcal{C}$ be a combinatorial regular building of finite thickness and denote by $X_K$ its $K$-realization.
The complexes $X_D$ and $X_B$ are quasi-isometric locally finite uniformly $\re$-acyclic complexes.
\end{prop}

\begin{proof}
Endowed with its $CAT(0)$-metric, the Davis realization $X_D$ of a building $\C$ is a uniformly $\re$-acyclic (even contractible) complex of bounded geometry. On the other hand, the Bestvina realization $X_B$ of $\C$ may be endowed with a length metric by defining a metric on the Bestvina chamber. By Proposition \ref{UniformlyAcyclic}, the space $X_B$ is uniformly $\re$-acyclic. 

Since $W$ acts properly and cocompactly on both Coxeter complexes $\S_D$ and $\S_B$, if $f : \S_D \to \S_B$ is a map sending a simplex in the chamber $w.D$ to some point in the simplex of $w.B$ of the same type, then $f$ is a quasi-isometry. This map can be defined in the same way between the corresponding buildings, thus we get a map $F: X_D \to X_B$ that restricts to a quasi-isometry on each apartment (whose constants do not depend on the choice of an apartment). Since retractions do not increase distances \cite[Proposition 12.18]{abramenko-brown}, a geodesic between two points is contained in an apartment for both realizations, thus $F$ is a quasi-isometric embedding. Quasi-surjectivity of $F$ follows since the same apartment system covers both buildings.
\end{proof}

\subsection{Non-vanishing of $\ell^p$-homology}

Now we show non-vanishing of top dimensional $\ell^p$-homology for large $p$ using the same idea as in the proof of Theorem \ref{criticalPM}.

\begin{thm}\label{BestvinaNonVanishing}
Let $X_B$ be the Bestvina realization of a regular building of type $(W,S)$ and thickness $\mathbf{q} +1$, with $\mathbf{q} \geq 2$. Let $d =  \dim {X_B}$. For all $p > 1 + e_\mathbf{q}(W)$, we have $\ell^p H_d(X_B) \neq \{0\}$.
\end{thm}

\begin{proof}
Let 
\begin{equation*}
    \t = \sum_{w \in W_\t} (-1)^{l(w)} w. \t_0 
\end{equation*}
be the nonzero locally finite cycle in $H_d^{\mathrm{lf}}(X_B; \re)$ given by Theorem \ref{BestvinaCycle}, for some well chosen non-trivial relative cycle $\t_0$ and let $W_\t$ be the parabolic subgroup of $W$ corresponding to the support of $\t$. The element $\rho^*(\t)$ is a nonzero locally finite cycle in $X_B$ and
\begin{equation*}
    || \rho^*(\t) ||_p^p \leq \sum_{w \in W_\t} \sum_{c \in \rho^{-1}(w)} \frac{||\t_0||_p^p}{q_w^{p}} = ||\t_0||_p^p \sum_{w \in W_\t} q_w^{1-p} = ||\t_0||_p^p W_\t(\mathbf{q}^{1-p}).
\end{equation*}
By Proposition \ref{GrowthConvergence}, the series $W_\t(\mathbf{q}^{1-p})$ converges for $p > 1 + e_\mathbf{q}(W_\t)$. 
The number $e_\mathbf{q}(W_\t)$ is finite since $\mathbf{q}\geq 2$. Thus $\rho^*(\t) \in \ell^p H_d(X_B)$ for $p > 1 + e_\mathbf{q}(W_\t)$.
\end{proof}

Since $\ell^p$-cohomology is invariant by quasi-isometries (see Theorem \ref{QIinvariance}), Proposition \ref{Quasi-isometry-XB-XD} implies that this result persists on the $\ell^p$-cohomology of the Davis realization of a building.

\begin{cor}\label{BestvinaNonvanishingCohomology}
Let $X_D$ be the Davis realization of a regular building of type $(W,S)$ and thickness $\mathbf{q} +1$, with $\mathbf{q} \geq 2$. Let $d = \vcd_\re(W)$. Then we have: \\
$\ell^p H^k(X_D) = \{0\}$ for all $k > d$ and $p >1$, \\ 
$\ell^p \overline{H^d}(X_D) \neq \{0\}$ for all $1 < p < 1 + e_\mathbf{q}(W)^{-1}$.
\end{cor}

\begin{rem}
Non-vanishing of $\ell^p H^d(X_D)$ in degree $d = \vcd_\re(W)$ can be obtained directly without invoking the Bestvina realization.
Indeed, we have 
\begin{equation*}
    d = \max \{ n , H_n(D, D^{T} ; \re) \neq 0 \textrm{ for some spherical } T \}
\end{equation*}
\cite[Corollary 8.5.5]{davisbook}. If we pick a non-trivial element in $ H_d(D, D^{T} ; \re)$, we can extend it to a cycle $\t \in Z_d^{\mathrm{lf}}(\S; \re)$ with bounded coefficients and non-trivial in $ H_d^{\mathrm{lf}}(\S; \re)$, as in the proof of \ref{BestvinaCycle}. Then the proof of Theorem \ref{BestvinaNonVanishing} says that the $\ell^p$-norm of the cycle $\rho^*(\t)$ on $X_D$ converges for $p > 1 + e_\mathbf{q}(W)$. If $\rho^*(\t)$ was trivial in $\ell^p H_d(X_D)$, then using that $\rho_*$ commutes with $\partial$, we see that $\t$ would be trivial inside $H_d^{\mathrm{lf}}(\S; \re)$.

The vanishing part of the corollary may seem trivial without using the Bestivna complex, but it is not if we do not have a big automorphism group acting on the building. This is the reason why we introduce the Bestvina complex.
\end{rem}

\section{Conformal dimension of Gromov-hyperbolic buildings}

In \cite{clais-conformal}, Clais obtains bounds for the conformal dimension of Gromov-hyperbolic buildings (of constant thickness) coming from right-angled Coxeter groups. These are formulated in terms of the conformal dimension of an apartment, the parameter of a visual metric and the combinatorial data of the building. Here, we generalize these bounds to arbitrary Gromov-hyperbolic buildings of finite thickness. Most of the proof remains the same, but instead of using combinatorial modulus techniques, we use a separation property of the first $\ell^p$-cohomology of hyperbolic groups.

\subsection{Conformal dimension and $\ell^p$-cohomology}

In this subsection let $X$ be a contractible, Gromov-hyperbolic simplicial complex with bounded geometry, let $\partial X$ be its Gromov boundary and $( \cdot | \cdot )_o$ be the Gromov product on $X$ with basepoint $o \in X$.

A metric $d$ on $\partial X$ is \textit{visual} if there exist constants $\l > 1$ and $C\geq 1$ such that for every $\xi, \eta \in \partial X$ we have:
\begin{equation*}
    C^{-1} \l^{- (\xi | \eta)_0} \leq d(\xi, \eta) \leq C \l^{- (\xi | \eta)_0}.
\end{equation*}
The number $\l > 1$ is the \textit{parameter} of the metric $d$. It was shown by Gromov that such a metric always exists for $\l > 1$ close enough to $1$ \cite[p. 435]{bridson-haefliger}. Moreover, by a theorem due to Coornaert \cite[5.4 and 7.4]{coornaert}, such a metric is \textit{Ahlfors-regular}, that is, there exist constants $C \geq 1$, $Q > 0$ and a measure $\mu$ on $\partial X$ such that for every ball $B(r) \subseteq (\partial X, d)$ of radius $r < \mathrm{diam}(\partial X, d)$, we have:
\begin{equation*}
    C^{-1} r^Q \leq \mu (B(r)) \leq Cr^Q.
\end{equation*}
This implies that $\mu$ is equivalent to the $Q$-Hausdorff measure of $(\partial X, d)$ and that $Q$ is the Hausdorff dimension of $(\partial X ,d)$.

We define the \textit{shadow} of a ball $B(x,R) \subset X$ as the subset of $\partial X$:
\begin{equation*}
    O(x,R) = \{ \xi \in \partial X, \, [1,\xi[  \,\cap \, B(x, R) \neq \emptyset \}.
\end{equation*}

We are ready to define the \textit{conformal gauge} $\mathcal{J}(\partial X)$ of $ \partial X$ as the set of metrics on $\partial X$ whose balls are similar to shadows of balls in $X$. More precisely:

\begin{defn}\cite[Chapter 15]{heinonen}
A metric $d$ on $\partial X$ is in the \textit{conformal gauge} $\mathcal{J}(\partial X)$ of $\partial X$ if it is Ahlfors-regular and if it satisfies the following two complementary conditions for every $R >0$ large enough:

$(i)$ There is an increasing function $\varphi: [1, + \infty [ \to [0 , +\infty [$ such that for every pair $B_1 \subseteq B_2$ of balls of radii $r_1$ and $r_2$, $r_1 \leq r_2$, there are shadows $O(x_1, R)$ and $O(x_2, R)$ satisfying $O(x_1, R) \subseteq B_1 \subseteq B_2 \subseteq O(x_2, R)$ and $|x_1 - x_2| \leq \varphi(\frac{r_2}{r_1})$.

$(ii)$ There is an increasing function $\psi: [0, + \infty [ \to [1 , + \infty [$ such that for every pair $O(x_1, R) \subseteq O(x_2, R)$ of shadows, there are balls $B_1$ and $B_2$ of radii $r_1$ and $r_2$, $r_1 \leq r_2$, satisfying $ B_1 \subseteq O(x_1, R)  \subseteq O(x_2, R)\subseteq   B_2 $ and $ \frac{r_2}{r_1} \leq \psi(|x_1 - x_2|)$.
\end{defn}

The conformal gauge is a complete quasi-isometry invariant of $X$. Indeed, every quasi-isometry $\phi: X \to Y$ defines a quasi-symmetric homeomorphism $\partial \phi : \partial X \to \partial Y$, that is, a homeomorphism such that the map sending $d \in \mathcal{J}(\partial Y)$ to the metric $d(\partial \phi(\cdot), \partial \phi(\cdot))$ on $\partial X$ is a bijection from $ \mathcal{J}(\partial Y)$ to $ \mathcal{J}(\partial X)$. Conversely, every homeomorphism from $\partial X$ to $\partial Y$ with this property is the extension on the boundary of a quasi-isometry from $X$ to $Y$ \cite{bonk-schramm}.

\begin{defn}\label{def-confdim}
The \textit{(Ahlfors-regular) conformal dimension} of $\partial X$ is:
\begin{equation*}
    \Cdim(\partial X) = \inf \{ \Hdim(\partial X,d) ; d \in \mathcal{J}(\partial X) \}.
\end{equation*}
\end{defn}

Now we present a characterization of the conformal dimension in terms of functions separating points of the boundary. Define the following subspace of $\ell^pH^1(X)$:
\begin{equation*}
    \ell^pH^1_{\mathrm{cont}}(X) = \{ [f] \in \ell^pH^1(X) \, | \, f \textrm{ extends continuously to } X^{(0)} \cup \partial X \}
\end{equation*}
and for $[f] \in \ell^pH^1_{\mathrm{cont}}(X)$, denote by $f_\infty$ the extension of $f$ to the boundary (by radial limit). We define the space of limit functions of $\ell^pH^1_{\mathrm{cont}}(X)$ by:
\begin{equation*}
    A_p(\partial X) = \{ u : \partial X \to \re \, | \, u = f_\infty \textrm{ for some } [f] \in \ell^pH^1_{\mathrm{cont}}(X) \}.
\end{equation*}

We state the main result we are going to use later. We will only need it in the case of Gromov-hyperbolic groups, so we state it only in this setting in order to have simple hypotheses.

\begin{thm}\label{separate-boundary} \cite[3.8]{bourdon-kleiner}
Let $\G$ be a Gromov-hyperbolic group with connected boundary $\partial \G$.
Then $A_p(\partial \G)$ separates points in $\partial \G$ if and only if $p > \Cdim(\partial \G)$.
\end{thm}

\begin{rem}
This result holds for any Gromov-hyperbolic group $\G$, without connectedness assumption. Indeed, connected components of $\partial \G$ are in bijection with ends of the group $\G$. Combining Stallings' theorem on ends of groups \cite{stallings-ends} and Dunwoody's accessibility theorem \cite{dunwoody-accessibility}, we obtain that if $\G$ is finitely presented and has more than one end, then $\G$ splits as the fundamental group of a finite graph of groups with finite edge groups, such that vertex groups $(\G_i)_{i \in I}$ are either finite or $1$-ended. It is known that the $\G_i$ are quasiconvex, thus hyperbolic \cite[1.2]{bowditch-splitting}, and that any connected component of $\partial \G$ can be identified either with the boundary $\partial \G_i$ for some 1-ended vertex group $\G_i$ or with a point in the boundary of a tree \cite[2.4]{carrasco-mackay}. 

With this we can show that $A_p(\partial \G)$ separates points in $\partial \G$ if and only if $A_p(\partial \G_i)$ separates points in $\partial \G_i$ for every 1-ended $\G_i$. The direct implication is trivial. For the converse, it is enough to show that we can separate two points in two different connected components by some element in $A_p(\partial \G)$. Denote by $Z$ the tree of spaces on which $\G$ acts by deck transformations, $T$ its associated Bass-Serre tree and $p: Z \to T$ the natural retraction. If $x, y \in \partial Z$ lie in two different connected components, there exists some edge $e$ of the tree such that any geodesic $\gamma$ going from $x$ to $y$ passes through the edge space $p^{-1}(e)$. The complement of $p^{-1}(e)$ in $Z$ has two connected components $Z_x$ and $Z_y$ (because $T$ is a tree). Let $f$ be the characteristic function of vertices of $Z_x$. Its differential is nonzero at most on edges of the space $p^{-1}(e)$, which is bounded. Thus $f \in \ell^p H^1(Z)$, the limit function $f_\infty$ is in $A_p(\partial \G)$ and separates $x$ and $y$.

By \ref{separate-boundary}, $A_p(\partial \G_i)$ separates points in $\partial \G_i$ for every 1-ended $\G_i$ if and only if $p >  \max_{i\in I} \Cdim(\partial \G_i)$. In \cite[6.2]{carrasco}, Carrasco shows that either all the $\G_i$'s are finite and $\Cdim(\partial \G) = 0$ or there is at least one 1-ended $\G_i$ and $\Cdim(\partial \G) = \max_{i\in I} \Cdim(\partial \G_i)$. 
\end{rem}

We finish this subsection by recalling a construction due to Elek of functions on $X^{(0)}$ starting from functions on $\partial X$ \cite{elek-lp-cohomology}. This also shows non-vanishing of the first $\ell^p$-cohomology of hyperbolic simplicial complexes for large $p> 1$.
Given a metric $d \in \mathcal{J}(\partial X)$ and a $d$-Lipschitz function $u : \partial X \to \re$, we define a function $\Phi(u) : X^{(0)} \to \re$ as follows. Fix $R > 0$ large enough and for every $x \in X^{(0)}$ choose some $\xi_x \in O_X(x, R)$. We set $\Phi(u)(x) = u (\xi_x)$ for every $x \in X^{(0)}$. In particular, we have $(\Phi(u))_\infty = u$.

\begin{prop} \label{Poisson-transform} \cite[2.3]{bourdon-survol}
Let $d \in \mathcal{J}(\partial X)$, $u : \partial X \to \re$ be a $d$-Lipschitz function and $p > \Hdim(\partial X, d)$. We have $||d\Phi(u)||_p < \infty$ and hence $[\Phi(u)] \in \ell^p H_\mathrm{cont}^1(X)$. Moreover, if $u$ is non-constant, $[\Phi(u)] \neq 0$ in $\ell^p H_\mathrm{cont}^1(X)$. In particular, $\ell^p H^1(X) \neq \{0\}$ for all $p > \Cdim(\partial X, d)$.
\end{prop}

We can think of this construction as a Poisson transform defined by integrating functions on the boundary over shadows with respect to a probability measure. Here we chose a Dirac mass on an element of the shadow. The choice of this particular probability measure is not important because the function we are integrating is Lipschitz.

\subsection{Upper bound of the conformal dimension of hyperbolic buildings}

Let $X$ be the Davis realization of a building of type $(W,S)$ and thickness $\mathbf{q} +1$. Denote by $\S$ the Davis complex of $(W,S)$. In this section we assume that $W$ is a Gromov-hyperbolic Coxeter group, so that $X$ and $\S$ are Gromov-hyperbolic spaces (and in fact they can be endowed with a $CAT(-1)$-metric by \cite{moussong-thesis}). The previous section applies in this context. 

First we obtain an upper bound for $\Cdim(\partial X)$. To do this, we bound the Hausdorff dimension of the visual metric associated to a natural distance defined using the underlying combinatorial chamber system $\C$.

We call $\mathcal{G}$ the dual graph of $X$, that is, $\mathcal{G}$ is the graph where vertices are chambers of $X$, and two vertices are joined by an edge when their corresponding chambers are adjacent. First, notice that $\mathcal{G}$ and $X$ are quasi-isometric, so that the boundaries $\partial \mathcal{G}$ and $\partial X$ are homeomorphic. This homeomorphism is also quasi-symmetric, so we can identify their conformal gauges.

We define a combinatorial distance on $\mathcal{G}$ using both the $W$-distance and the thickness of $X$. Namely, for $x, y \in \mathcal{G}^{(0)}$, set $|x-y|_\mathbf{q} = \log q_{d_W(x,y)}$. In other words, if $s_1 \ldots s_l$ is a reduced expression for $w = d_W(x, y)$, then $|x - y|_\mathbf{q} = \log q_{s_1} + \ldots + \log q_{s_l}$. If the thickness is constant and equal to $q + 1$, this is just $|x-y|_q =  l(w) \log q$. This distance on $\mathcal{G}$ induces a visual metric $d_\mathbf{q}$ on $\partial \mathcal{G}$ of parameter $\l(\mathbf{q})$. We fix an origin $o \in \mathcal{G}^{(0)}$ and write $|x|_\mathbf{q} := |x-o|_\mathbf{q}$.

The distance $| \cdot - \cdot |_\mathbf{q}$ naturally restricts to the dual graph of $\S$, and induces a visual metric on $\partial \S$ with the same parameter $\l(\mathbf{q})$, which we also denote by $d_\mathbf{q}$ and which is also the restriction of the visual metric $d_\mathbf{q}$ on $\partial X$ to the boundary $\partial \S$. From \cite[Corollaire 7.6]{coornaert}, we know that  \begin{equation*}
    \Hdim(\partial \S, d_\mathbf{q}) = \frac{e_\mathbf{q}(W)}{\log \l(\mathbf{q})}.
\end{equation*}

The following upper bound for the Hausdorff dimension of $(\partial \mathcal{G}, d_\mathbf{q})$ can be seen as a thickened version of Coornaert's formula.

\begin{prop}\label{upper bound}  The visual metric $d$ satisfies:
\begin{equation*}
    \Hdim(\partial \mathcal{G}, d_\mathbf{q}) \leq  \Hdim(\partial \S, d_\mathbf{q}) (1 + e_\mathbf{q}(W)^{-1} ).
\end{equation*}
\end{prop}

\begin{proof}
Consider the combinatorial sphere $\mathcal{G}_n = \{ x \in \mathcal{G}^{(0)} \, | \, n \leq |x|_\mathbf{q} < n+1 \}$. Fix $R > 0$ bigger than the diameter of all simplices in $\mathcal{G}$. Cover the boundary by the shadows $O_\mathcal{G}(x, R)$ for $x \in \mathcal{G}_n$. Since the metric $d$ lies in the conformal gauge of $ \partial \mathcal{G}$, the shadows $O_\mathcal{G}(x, R)$ are similar to balls $B(\xi, r)$ in $(\partial \mathcal{G}, d_\mathbf{q})$ of radius $r = \l(\mathbf{q})^{ - |x|_\mathbf{q}}$. Denote by $\mu_\a$ the $\a$-Hausdorff measure of $(\partial \mathcal{G}, d_\mathbf{q})$. Thus, there is a constant $C\geq 1$ such that for every $x\in \mathcal{G}^{(0)}$:
\begin{equation*}
    C^{-1} \l(\mathbf{q})^{ - |x|_\mathbf{q} \alpha} \leq \mu_\a(O_\mathcal{G}(x, R)) \leq C \l(\mathbf{q})^{ - |x|_\mathbf{q} \alpha}.
\end{equation*}
For fixed $\a$, we compute the $\a$-Hausdorff measure of the cover $\{ O_\mathcal{G}(x, R), x \in \mathcal{G}_n \}$ and estimate its limit when $n \to \infty$. We set $W_n = \{ w \in W \, | \, n \leq \log q_w < n+1 \}$. 
We have:
\begin{align*}
    \mu_\a(\partial X) & \leq \sum_{x \in \mathcal{G}_n} \mu_\a(O_\mathcal{G}(x, R)) \leq  C \sum_{x \in \mathcal{G}_n} \l(\mathbf{q})^{ - |x|_\mathbf{q} \alpha} \\
    & =    C \sum_{w \in W_n} q_w \l(\mathbf{q})^{ - |w|_\mathbf{q} \alpha} 
    = C \sum_{w \in W_n} q_w^{ 1 - \log \l(\mathbf{q}) \alpha}.
\end{align*}
By Proposition \ref{GrowthConvergence}, this sum tends to $0$ when $\log \l(\mathbf{q}) \alpha - 1 > e_\mathbf{q}(W)$ (because the associated series on $n$ converges). This means that $\mu_\a(\partial X) = 0$ when $ \a > \frac{1}{\log \l(\mathbf{q})} (1 + e_\mathbf{q}(W) )$. This term is just $\Hdim(\partial \S, d_\mathbf{q}) (1 + e_\mathbf{q}(W)^{-1} )$ by Coornaert's formula.
\end{proof}

\begin{rem}
When the thickness vector $\mathbf{q}+1$ is constant and equal to $q +1$, we can choose the parameter $\l(\mathbf{q})$ of the metric $d_\mathbf{q}$ so that we can describe precisely its dependence on $q$, allowing us to express the right hand side of the inequality in Proposition \ref{upper bound} in a more explicit way. Indeed, we may choose the combinatorial distance on $\mathcal{G}$ defined by $|x-y| = l(d_W(x, y))$ and our assumption gives the equality $| \cdot - \cdot |_\mathbf{q} = \log q | \cdot - \cdot | $. Call $\l$ the parameter of some visual metric $d$ on $\partial \mathcal{G}$ associated to the combinatorial distance $| \cdot - \cdot |$. The same metric $d$ is also visual for $| \cdot - \cdot |_\mathbf{q}$, so we can take $d_\mathbf{q}= d$ and its parameter $\l(q)$ (associated to $| \cdot - \cdot |_\mathbf{q}$) is given by
\begin{equation*}
    \l(q) = \l ^{1/ \log q}.
\end{equation*} 
Hausdorff dimensions depend only on the resulting metric and not on the parameters, so the bound in the proposition can be expressed as:
\begin{equation*}
     \Hdim(\partial \mathcal{G}, d) \leq \Hdim(\partial \S, d) ( 1 + \frac{\log q}{e(W)})) = \frac{e(W)}{\log \l} ( 1 + \frac{\log q}{e(W)}),
\end{equation*}
where $e(W)$ is the exponential growth rate of the group $(W, S)$ with respect to word length. 
\end{rem}

\subsection{Lower bound of the conformal dimension of hyperbolic buildings}

In this section we obtain a lower bound for the conformal dimension $\Cdim(\partial X)$ in terms of the conformal dimension $\Cdim(\partial \S)$ and the growth of the Weyl group $(W, S)$, weighted by the thickness $\mathbf{q} +1$. We keep the same notations from the previous section. 

\begin{rem} Theorem \ref{BestvinaNonvanishingCohomology} combined with \cite[Théorème A]{bourdon-lp-degre-superieur} already gives a lower bound for $\Cdim (\partial X)$:
\begin{equation*}
    \frac{\Cdim (\partial X)}{\vcd_\re(W) -1} \geq  1 + e_\mathbf{q}(W)^{-1}.
\end{equation*}
By \cite{bestvina-mess}, we know that $ \vcd_\re(W) -1 \leq \mathrm{Topdim}(\partial \S) \leq \Cdim (\partial \S)$. In this section, we obtain a sharper inequality by replacing $\vcd_\re(W)-1$ by $\Cdim (\partial \S)$.
\end{rem}

Fix a retraction $\rho: X \to \S$ onto an apartment $A$. We identify the apartment $A$ with the abstract Coxeter complex $\S$ and the boundary $\partial A$ with $\partial \S$. For a map $f : \S^{(0)} \to \re$, we define its pullback by the retraction $\rho$ by $\rho^*f = f \circ \rho$. 

In a similar way, we can pushforward functions on $X^{(0)}$ through the retraction $\rho$, by averaging over preimages. For a map $h : X^{(0)} \to \re$, we define for $x \in \S^{(0)}$,
\begin{equation*}
    (\rho_* h)(x) = \frac{1}{|\rho^{-1}(x)|} \sum_{y \in \rho^{-1}(x)} h(y).
\end{equation*}

We want to define a pushforward for functions on the boundary. For this, let $\mathcal{A}_o$ be the set of apartments of $X$ containing an origin $o$ lying in the interior of the central Davis chamber of the retraction $\rho$. We endow $\mathcal{A}_o$ with a probability measure $\nu$ as in \cite[Section 2.2.3]{bourdon-fuchsien}, the following proposition sums up its main properties.
\begin{prop}\cite[2.2.3 and 2.2.4]{bourdon-fuchsien}
For $y \in X^{(0)}$, let $\mathcal{A}_y$ be the set of apartments containing $o$ and $y$. There exists a Borel probability measure $\nu$ on $\mathcal{A}_o$ such that:
\begin{equation*}
    \nu(\mathcal{A}_y) = |\rho^{-1}(\rho(y))|^{-1}.
\end{equation*}
\end{prop}

We view an apartment in $\mathcal{A}_o$ as an isometric embedding $p : \S \to X$. This induces a map $\partial p : \partial \S \to \partial X$. Given $d\in \mathcal{J}(\partial X)$, we denote by $\mathrm{Lip}(\partial X, d)$ the space of Lipschitz functions on $(\partial X, d)$. For $u \in \mathrm{Lip}(\partial X, d)$, define the function $\partial \rho_* (u)$ on $\partial \S$ by:
\begin{equation*}
    \partial \rho_* (u) (\eta) = \int_{\mathcal{A}_o} u(\partial p(\eta)) \mathrm{d}\nu(p).
\end{equation*}

The function $\partial \rho_* (u)$ is continuous and the pushforward $\partial \rho_*$ factors through $\rho_*$:

\begin{lem}\label{compatiblity-Poisson}
The following diagram commutes:
\[ \begin{tikzcd}
\ell^pH^1_{\mathrm{cont}}(X) \arrow[swap]{d}{\rho_*} & \mathrm{Lip}(\partial X, d) \arrow{l}{\Phi} \arrow[swap]{d}{\partial \rho_*} \\%
\ell^pH^1_{\mathrm{cont}}(\S) \arrow{r}{(\cdot)_\infty} & \mathcal{C}(\partial \S)
\end{tikzcd}
\]
\end{lem}

\begin{proof}
Let $v \in \mathrm{Lip}(\partial X, d)$. For $y \in X^{(0)}$, let $\mathcal{A}_y$ be the set of apartments containing $o$ and $y$. Notice that the set $\mathcal{A}_o$ is the disjoint union of the sets $\mathcal{A}_y$ for $y$ ranging over a given fiber of the retraction $\rho$ and recall that $\nu(\mathcal{A}_y) = |\rho^{-1}(\rho(y))|^{-1} $ \cite[Lemme 2.2.4]{bourdon-fuchsien}. Thus:
\begin{equation*}
    \rho_* \Phi(v)(x) = \frac{1}{|\rho^{-1}(x)|} \sum_{y \in \rho^{-1}(x)} \Phi(v) (y) = \int_{\mathcal{A}_o} \Phi(v) (p (x)) \mathrm{d} \nu (p).
\end{equation*}

Since $v \in \mathrm{Lip}(\partial X, d)$, we know that for $\xi \in \partial X$, we have $\lim_{y \to \xi} \Phi(v) (y) = v (\xi)$. From this and the dominated convergence theorem we first have for $\eta \in \partial \S$:
\begin{equation*}
    \lim_{x \to \eta} \rho_* \Phi(v) (x)  =
    \int_{\mathcal{A}_o} \lim_{x \to \eta} \Phi(v) (p (x)) \mathrm{d} \nu (p) =  \int_{\mathcal{A}_o} v(\partial p(\eta)) \mathrm{d}\nu(p) = \partial \rho_* (u) (\eta).
\end{equation*}
\end{proof}

We are ready to state the main result of this section.

\begin{thm}\label{ConfdimBounds} Let $(W,S)$ be a Gromov-hyperbolic Coxeter system, $\S$ the Davis complex of $(W,S)$ and $X$ the Davis realization of a building with Weyl group $(W,S)$ and thickness $\mathbf{q} +1$. Let $d_\mathbf{q}$ be a visual metric on $\partial \S$ induced by the distance $| \cdot - \cdot|_\mathbf{q}$.
We have:
\begin{equation*}
     \mathrm{Confdim}(\partial \S) (1 + e_\mathbf{q}(W)^{-1}) \leq \mathrm{Confdim}(\partial X) \leq \Hdim(\partial \S, d_\mathbf{q}) (1 + e_\mathbf{q}(W)^{-1}).
\end{equation*}
\end{thm}

\begin{rem}
1. When $X$ is a Fuchsian building, the boundary $\partial \S$ is a circle and thus $\Cdim(\partial \S) = 1$. In \cite{bourdon-fuchsien}, Bourdon constructs a visual metric $d_\mathbf{q}$ with parameter $\l(\mathbf{q}) = \exp{{e_\mathbf{q}(W)}}$, so that Coornaert's formula gives $\Hdim(\partial \S, d_\mathbf{q}) = 1$. Therefore, the inequalities of the theorem are optimal for these buildings and give a slightly different proof from that of \cite{bourdon-fuchsien} of the equality:
\begin{equation*}
    \Cdim(\partial X) = 1 + e_\mathbf{q}(W)^{-1}.
\end{equation*}

2. In the case of constant thickness the inequalities can be expressed as:
\begin{equation*}
     \mathrm{Confdim}(\partial \S) (1 + \frac{\log q}{e(W)}) \leq \mathrm{Confdim}(\partial X) \leq  \Hdim(\partial \S, d) ( 1 + \frac{\log q}{e(W)}))
\end{equation*}
where $d$ is the visual metric on the boundary of $\S$ equipped with the length function. These are the inequalities obtained by Clais in \cite{clais-conformal} for some hyperbolic buildings coming from right-angled Coxeter groups. In particular, they imply that $\Cdim(\partial X)$ grows as $\log q$ as a function on $q$, in the sense that there are constants $C_1, C_2>0$ such that for all $q \geq 2$ we have $C_1 \log q \leq \Cdim(\partial X) \leq C_2 \log q $.
\end{rem}

\begin{proof}[Proof of Theorem \ref{ConfdimBounds}]
The upper bound was already shown in Proposition \ref{upper bound}.

We show the lower bound.
Choose a metric $d \in \mathcal{J}(\partial X)$ and a family $\mathcal{V}$ of Lipschitz functions (for $d$) on $\partial X$ such that the family $\mathcal{U} = \partial \rho_* \mathcal{V}$ of averages of functions in $\mathcal{V}$ separates points in $\partial \S$. In the remark after \ref{separate-boundary} we saw that $\Cdim(\partial \S)$ is the infimum of all $p$'s such that $A_p(\partial \S)$ separates points in $\partial \S$. By Lemma \ref{compatiblity-Poisson}, the limit functions of the family $\rho_* \Phi(\mathcal{V})$ are the functions in $\mathcal{U}$. Thus, for every $\varepsilon >0$, there exists $v \in \mathcal{V}$ such that the function $f = \rho_* \Phi (v)$ satisfies:
\begin{equation*}
    \inf\{ p \geq 0, ||df||_p < \infty \} \geq \Cdim(\partial \S) - \varepsilon.
\end{equation*}
Now we consider $\rho^*(f)$, its differential satisfies for each edge $[x,y] \in \S^{(1)}$:
\begin{equation*}
    d(\rho^*f)([x,y]) = df(\rho([x, y])).
\end{equation*}
Thus, the $\ell^r$-norm of its differential is given by:
\begin{equation*}
    ||d(\rho^*f)||_r^r = \sum_{\s \in \S^{(1)}} |\rho^{-1}(\s)| |df(\s)|^r.
\end{equation*}
Define the function:
\begin{equation*}
    P_f(s) = \inf \{ r > 0, \sum_{\s \in \S^{(1)}} |\rho^{-1}(\s)|^s |df(\s)|^r < \infty \}
\end{equation*}
and $P_f(s) = + \infty$ if the corresponding set is empty.

\begin{claim} We have: $(1 + e_\mathbf{q}(W)^{-1}) (\Cdim (\partial \S) - \varepsilon) \leq P_f(1) \leq \Hdim(\partial X, d)$.
\end{claim}

The theorem follows from this claim. Indeed, since it holds for any metric $d \in \mathcal{J}(\partial X)$ and for any $\varepsilon >0$, we obtain the lower bound of the theorem. We now prove separately each of the two inequalities in this claim. \smallskip

\noindent \textit{Proof of the upper bound of the Claim.}
We know that $||d(\Phi(v))||_p < \infty$ for $p > \Hdim(\partial X, d)$ and by definition of $f$ we know that $ \rho^*f = \rho^* \rho_* (\Phi (v))$. By Jensen's inequality, averaging a function on $X^{(0)}$ over fibers of $\rho$ reduces the $\ell^p$-norm of the differential, that is:
 
 \begin{lem}
 For any $h: X^{(0)} \to \re$ and $p > 1$ we have $||d(\rho^* \rho_* h) ||_p \leq ||d h||_p$.
 \end{lem}
 \begin{proof}
Denote by $E$ the operator $\rho^* \rho_*$. This operator can also be defined on functions $X^{(1)} \to \re$. We first show that the coboundary operator $d$ commutes with $E$. For $x \in X^{(0)}$ we have $E h (x) = \frac{1}{| \rho^{-1}(\rho(x)) |} \sum_{y \in \rho^{-1}(\rho(x)) } h(y)$. Let $\s = [x, x'] \in X^{(1)}$. Notice that:
\begin{equation*}
    | \rho^{-1}(\rho(\s)) | = | \rho^{-1}(\rho(x)) | | \{\t \in \rho^{-1}(\rho(\s)) , x \in \t \} |.
\end{equation*} 
Thus we have:
\begin{align*}
    d(E h)(\s) & = \frac{1}{| \rho^{-1}(\rho(x)) |} \sum_{y \in \rho^{-1}(\rho(x)) } h(y) - \frac{1}{| \rho^{-1}(\rho(x')) |} \sum_{y' \in \rho^{-1}(\rho(x')) } h(y') \\
    & = \frac{1}{| \rho^{-1}(\rho(\s)) |} \sum_{[y, y'] \in \rho^{-1}(\rho(\s)) } (h(y) - h(y')) = E(dh) (\s).
\end{align*}
Now Jensen's inequality, used as in Proposition \ref{Jensen}, gives directly the result:
\begin{equation*}
    ||d(Eh)||_p = ||E(dh)||_p \leq ||dh||_p.
\end{equation*}
 \end{proof}
 
Thus we have $||d(\rho^*f) ||_p \leq ||d(\Phi (v))||_p < \infty$ for $p > \Hdim(\partial X, d)$. \smallskip
 
\noindent \textit{Proof of the lower bound of the Claim.}
Our goal is to obtain a lower bound for $P_f(1)$. This will be done using the following convexity lemma, also used in \cite{clais-conformal}.

\begin{lem}
The function $P_f$ is convex on $\re$.
\end{lem}

\begin{proof}
We have to check that for $t \in [0,1]$, one has:
\begin{equation*}
    P_f(ta + (1-t)b) \leq t P_f(a) + (1-t)P_f(b).
\end{equation*}
If $P_f(a) = + \infty$ or $P_f(b) = + \infty$, there is nothing to prove. Otherwise, we have to check that for $t \in ]0,1[$ and for every $\varepsilon >0$, one has:
\begin{equation*}
    \sum_{\s \in \S^{(1)}} |\rho^{-1}(\s)|^{ta + (1-t)b} |df(\s)|^{t P_f(a) + (1-t)P_f(b) + \varepsilon} < \infty.
\end{equation*}
By definition of the function $P_f$, we know that the series 
\begin{equation*}
    \sum_{\s \in \S^{(1)}} |\rho^{-1}(\s)|^{a} |df(\s)|^{P_f(a) + \varepsilon} \textrm{ and } \sum_{\s \in \S^{(1)}}  |\rho^{-1}(\s)|^{b} |df(\s)|^{ P_f(b) + \varepsilon} 
\end{equation*}
converge. Now the lemma follows from Hölder's inequality:
\begin{align*}
    & \sum_{\s \in \S^{(1)}} |\rho^{-1}(\s)|^{ta + (1-t)b} |df(\s)|^{t P_f(a) + (1-t)P_f(b) + \varepsilon} \\
    & = \sum_{\s \in \S^{(1)}}  \big( |\rho^{-1}(\s)|^{ta} |df(\s)|^{t P_f(a) + t \varepsilon} \big) \big( |\rho^{-1}(\s)|^{(1-t)b} |df(\s)|^{ (1-t)P_f(b) + (1-t) \varepsilon} \big) \\
    & \leq \Big(  \sum_{\s \in \S^{(1)}} |\rho^{-1}(\s)|^{a} |df(\s)|^{P_f(a) + \varepsilon} \Big)^t \Big( \sum_{\s \in \S^{(1)}}  |\rho^{-1}(\s)|^{b} |df(\s)|^{ P_f(b) + \varepsilon} \Big)^{1-t} \\
    & < \infty.
\end{align*}
\end{proof}
\noindent \textit{Proof of the lower bound of the claim (concluded).}  We define:
\begin{equation*}
    s_0 = \sup \{ s < 0,  \sum_{\s \in \S^{(1)}} |\rho^{-1}(\s)|^s < \infty \} < 0.
\end{equation*}
This number is defined so that for all $s < s_0$, we have $P_f(s) = 0$. Since $P_f$ is continuous, we have $P_f(s_0) = 0$. If $\s \in \S^{(1)}$ and if $w$ denotes the $W$-distance from the origin $o$ to the closest chamber containing $\s$, then $ |\rho^{-1}(\s)| = q_w$, so by Proposition \ref{GrowthConvergence} the sum appearing in the definition of $s_0$ converges for $s < -e_\mathbf{q}(W)$ and diverges for $s > -e_\mathbf{q}(W)$. Thus $s_0 = - e_\mathbf{q}(W)$.

We already showed that $P_f(1)< \infty$ and $P_f(s_0) < \infty$, thus $P_f(s) < \infty$ for all $s \leq 1$. Moreover, recall that $f$ was chosen so that $P_f(0) \geq \Cdim(\partial \S) - \varepsilon$.

By convexity of the function $P_f$, we have for $t < 0$:
\begin{equation*}
    P_f(t s_0) \geq (1 - t) P_f(0) + t P_f(s_0) = (1 - t) P_f(0).
\end{equation*}
In particular, for $t = s_0^{-1}$, we obtain:
\begin{equation*}
    P_f(1) \geq (1 + e_\mathbf{q}(W)^{-1}) P_f(0) \geq (1 + e_\mathbf{q}(W)^{-1}) (\Cdim (\partial \S) - \varepsilon). 
\end{equation*}

 \end{proof}

\bibliographystyle{amsalpha}
\bibliography{refs.bib}

\noindent Antonio López Neumann \\
Mathematical Institute of the Polish Academy of Sciences (IMPAN), Warsaw \\ 00-656 Warsaw, Poland \\
alopez@impan.pl \\

\end{document}